\documentclass[11pt]{article}

\usepackage{amsfonts,amsthm,empheq}

\usepackage{a4wide}
\usepackage{graphicx}      										% Grafiken einbinden
\usepackage{cite}
\usepackage[colorlinks=true,linkcolor=blue,citecolor=red]{hyperref}
\usepackage[inline]{enumitem}

\def\ind{\hspace{0.25in}}
\def\fP{\em \sf}
\def\cir{\makebox[0.5cm]{$\circ$}}
\def\IR{\mathop{\mbox{\rm Sround}}\nolimits}
\newcommand{\intg}[1]{\lfloor #1\rfloor} 

\newtheorem{lemma}{Lemma}[]
\newtheorem{theorem}{Theorem}[]
\newtheorem{remark}{Remark}[]

\newtheorem{alg}{Algorithm}[]
\newcommand{\R}{\ensuremath{\mathbb{R}}}

\newcommand{\g}{q}
\newcommand{\h}{\tilde f}
\newcommand{\VV}{\mathbb Var }
\newcommand{\CC}{\mathbb Cov}
\renewcommand{\epsilon}{\varepsilon}
\renewcommand{\phi}{\varphi}

\begin{document}
	
	%-- TITEL ----------------------------------------------------------%
	\title{Mean-field control variate methods for kinetic equations with uncertainties and applications to socio-economic sciences}
	\author{Lorenzo Pareschi\thanks{Department of Mathematics and Computer Science, University of Ferrara, Italy (\texttt{lorenzo.pareschi@unife.it})},\qquad
	 Torsten Trimborn\thanks{Institut f{\"u}r Geometrie und Praktische Mathematik, RWTH Aachen University, Germany (\texttt{trimborn@igpm.rwth-aachen.de})},\qquad
	 Mattia Zanella\thanks{Department of Mathematics "F. Casorati", University of Pavia, Italy (\texttt{mattia.zanella@unipv.it})}}
	\maketitle
	
\begin{abstract}
In this paper, we extend a recently introduced multi-fidelity control variate for the uncertainty quantification of the Boltzmann equation to the case of kinetic models arising in the study of multiagent systems. For these phenomena, where the effect of uncertainties is particularly evident, several models have been developed whose equilibrium states are typically unknown. In particular, we aim to develop efficient numerical methods based on solving the kinetic equations in the phase space by Direct Simulation Monte Carlo (DSMC) coupled to a Monte Carlo sampling in the random space. To this end, exploiting the knowledge of the corresponding mean-field approximation we develop novel mean-field Control Variate (MFCV) methods that are able to strongly reduce the variance of the standard Monte Carlo sampling method in the random space. We verify these observations with  several numerical examples based on classical models , including wealth exchanges and opinion formation model for collective phenomena. 
 \smallskip
 
\end{abstract} 

\tableofcontents

\section{Introduction}
In recent years kinetic theory emerged as a sound theoretical framework to describe a wide range of collective phenomena in socio-economy and life sciences for systems composed by a sufficiently large number of agents \cite{toscani2006kinetic,cordier2005kinetic,cordier2009mesoscopic,during2007hydrodynamics,during2018kinetic,furiolifokker,carrillo2010particle, trimborn2017portfolio}. For an introduction to these topics we refer to the recent surveys and monographs \cite{pareschi2013interacting,naldi2010mathematical,bellomoactive1, bellomoactive2}.   

Nevertheless, the design of realistic models for the description of human behavior has to face the lack of first principles and the dynamics are often inferred from empirical observations based on experimental data \cite{Parisi_etal,BFHM,HTVZ}. Especially in the field of socio-economic applications the precise form of the microscopic interactions is largely unknown. Thus, one typically constructs microscopic social forces which are able to qualitatively fit the macroscopic behavior of the system. In the context of kinetic modelling, this issue can be translated in structural uncertainties present both in initial observations and interaction rules, which can be considered in the form of uncertain parameters of the model depending on random quantities. Hence, the quantification of uncertainties in such social microscopic interactions on the observable behavior is of major importance. 

The introduced uncertainties inevitably increase the dimensionality of the problems. Therefore, we need to develop new numerical methods to efficiently quantify the impact of unknown quantities on the overall dynamics.  Among the various methods for uncertainty quantification we find intrusive stochastic Galerkin (sG) methods which provide spectral convergence towards the solution of the problem under suitable regularity assumptions \cite{HJ,poette2009uncertainty,zhu2017multi}. Beside sG methods we find non-intrusive approaches for UQ which does not require strong modification of the numerical scheme for the deterministic problem, like stochastic collocation methods. Such methods are non-intrusive, easy to parallelize \cite{XH} and, in principle, do not require any knowledge of the probability distribution of the uncertain parameter. 
We refer the interested reader to \cite{dimarco2017uncertainty,HJ,GJL,jin2018uncertainty,pareschi_review,pettersson2015polynomial,PWG,xiu2010numerical,jingwei2020}  for an additional overview on numerical methods for uncertainty quantification of hyperbolic and kinetic equations. 

In this work we concentrate on a non-intrusive approach based on a stochastic Monte Carlo (MC) sampling for Boltzmann-type equations. In in comparison to sG methods, techniques based on MC sampling  have a lower impact on the curse of dimensionality \cite{caflisch1998monte,LZ, giles2015multilevel, mishra2012multi, pareschi2013interacting}.
In details, following the methodology recently introduced in \cite{dimarco2018multi, dimarco2020multiscale}
for rarefied gas dynamics, we develop a variance reduction method based on a control variate approach which exploits a micro-macro-type decomposition of the solution. In the simplest setting proposed in \cite{dimarco2018multi} as a surrogate model to reduce the variance of the MC estimator, the corresponding Maxwellian steady state solution has been used. In many socio-economic applications, however, the equilibrium states of the Boltzmann models are unknown and therefore it is an open problem the determination of a suitable surrogate model that can be used as control variate.

In the present paper, we propose to use as surrogate model the corresponding mean field model of Fokker-Planck type obtained as an approximation of the original Boltzmann model in a grazing-type limit \cite{pareschi2013interacting, villani1998new}. 
In fact there are several reasons for adopting this control variate. First, if the mean-field model's steady state is known, we can use it directly as a control variate. Second, more in general, we can use the whole solution of the Fokker-Planck model as a control variate. This is possible thanks to the recently introduced structure preserving schemes for Fokker-Planck equations that preserve exactly the steady state of the equation \cite{pareschi2018structure}. The latter choice can be generally applied whereas for the former the steady state of the mean field model needs to be known. We call this novel variance reduction method, mean field control variate (MFCV) method. 

Furthermore, our mean field control variate approach is not based on a deterministic solver of the Boltzmann model but on a Direct Simulation Monte Carlo (DSMC) solver. This is beneficial since DSMC solver are widely used in order to solve Boltzmann type equations \cite{cercignani2000rarefied, pareschi2001introduction, pareschi2013interacting} and can be easily applied to generalized Boltzmann model simply by exploiting the rules defining the microscopic dynamics. This is an additional difficulty since our method needs to couple the deterministic solution of the mean field model with the DSMC solver of the Boltzmann equation. In this respect, our MFCV method can be regarded as a hybrid method in the phase space. 
The crucial assumptions of the MFCV method are the following:  First, the mean field model has to be less expensive to solve than the original kinetic model. Secondly, the collision regime of the space homogeneous Boltzmann model needs to be sufficiently close to the grazing collision regime.  

The rest of the manuscript is organized as follows: In the next section we introduce a general Boltzmann model well suited for socio-economic applications. Furthermore, we recall the derivation of the corresponding mean field model and discuss some examples. In Section \ref{prel} we introduce the basics of DSMC methods and shortly discuss the MC sampling method for uncertainty quantification. Then we present in Section \ref{MFCV} the novel MFCV method and emphasize its potential advantages in reducing the variance of the estimator used for uncertainty quantification. Several examples in the subsequent section demonstrate the advantages of the MFCV method in comparison to the standard MC technique. We finish this study with a short discussion of the present method and comment on further developments.

\section{One-dimensional kinetic models with uncertainties}\label{model}

Let us consider a general binary interaction model with uncertain mixing  \cite{dimarco2017uncertainty, tosin2017boltzmann}.  The pair of interacting agents is characterized by the pre-interaction states $v,w\in V \subseteq \mathbb R$ and the post-interaction states $v^{\prime}, w^{\prime} \in V$ are obtained as follows
\begin{equation}
\label{micro}
\begin{split} 
v^{\prime} &= v + \epsilon [(p_1(z)-1) v+ q_1(z)w] +D(v,z)\ \eta_\epsilon,\\
w^{\prime} &= w + \epsilon [p_2(z) v+ (q_2(z)-1) w] +D(w,z)\ \eta_\epsilon,
\end{split}
\end{equation}
where $\epsilon >0$ is a given constant, $p_i,q_i$,  $i=1,2,$ are suitable interaction functions depending on a random variable $z \in \Omega \subseteq \mathbb R^{d_z}$, $d_z\geq 1$. Furthermore, $\eta_\epsilon$ is a random variable with zero mean and variance $\sigma^2_\epsilon$ and the function $D(\cdot,z)$ is the local relevance of the diffusion. We will assume that the post-interaction states $w^{\prime}, v^{\prime}$ remain in the set $V$ up to the introduction of suitable conditions on $\eta_\epsilon$, see \cite{toscani2006kinetic}.

To describe the evolution of a large system of agents undergoing binary interaction we adopt a kinetic approach. Hence, we introduce the distribution function $f = f(t,w,z)$, such that $f(t,w,z)dw$ is the fraction of agents, at time $t\ge 0$, characterized by a state comprised between $w$ and $w + dw$ and parametrized by the uncertainty $z$. The evolution of $f$ is given in terms of the Boltzmann-type model for Maxwellian interactions
\begin{equation}\label{BoltzmannModel}
\begin{split}
&\frac{d}{dt} \int_V \varphi(w) f(t,w,z) dw\\
&\qquad = \dfrac{1}{\epsilon} \left\langle\iint_{V^2}  (\varphi(w^{\prime})-\varphi(w))\ f(t,w,z) f(t,v,z) dwdv \right \rangle,
\end{split}
\end{equation}
being $\varphi:V\rightarrow \mathbb R$ any observable quantity which may be expressed as a function of the microscopic state $w$ of the agents. The symbol $\left\langle \cdot \right\rangle$ denotes the expectation with respect to $\eta_\epsilon$. The model \eqref{BoltzmannModel} can be also   complemented with uncertainties on the initial condition  $f(0,w,z)=f_0(w,z)$. 

We can recast model \eqref{BoltzmannModel} in symmetric form as follows
\begin{equation}\label{SymBoltzmannModel}
\begin{split}
&\frac{d}{dt} \int_V f(t,w,z)\ \varphi(w) dw \\
&\qquad= \frac{1}{2\epsilon} \left\langle\iint_{V^2} (\varphi(w^{\prime}) +\varphi(v^{\prime})-\varphi(w)-\varphi(v)) f(t,w,z) f(t,v,z)    dwdv \right\rangle
\end{split}
\end{equation}

Taking $\varphi(w)=1$ is easily seen that the number of agents is conserved in time.
The evolution of the mean is obtained for $\varphi(w) = w$ which gives
\[
\dfrac{d}{dt} m_f(t,z) = \dfrac{1}{2} \iint_{V^2} [(p_1(z) + p_2(z)-1)v + (q_1(z)+q_2(z)-1)w]f(t,w,z)f(t,v,z)dw\,dv,
\]
and the mean is conserved for $p_1(z)+p_2(z) = 1$, $q_1(z)+q_2(z)=1$, indeed at the microscopic level the binary mean is conserved provided 
$$
\langle w^{\prime}+ v^{\prime}\rangle = v+w + \epsilon [(p_1(z)+p_2(z)-1)v + (q_1(z)+q_2(z)-1)w] = v+w.
$$
Notice that any nonconserved moment of the distribution function $f$ explicitly depends on $z$.

\subsection{Mean-field approximation}
An extensive qualitative study of the previously introduced Boltzmann model is very challenging. In particular, the asymptotic behavior of socio-economic Boltzmann models like \eqref{BoltzmannModel} are unknown. Therefore we employ the quasi-invariant interaction limit  \cite{cordier2005kinetic,pareschi2006self,toscani2006kinetic} to derive a mean-field approximation of the Boltzmann model. 

In kinetic theory a mean field model can be derived by the grazing limit \cite{FPTT2012,pomraning1992fokker,villani1998new} of the Boltzmann equation. The main idea is to introduce a scaling parameter such that we have a small change in the kinetic variable after each interaction, while keeping the macroscopic properties of the model unchanged. 

Let us introduce a time scaling parameter $\epsilon>0$ and define
\begin{equation}
\tau=\epsilon  t,\quad  f_\epsilon(\tau,w,z)=f(\tau/\epsilon,w,z). 
\end{equation}
Then, the distribution $ f_\epsilon$ is solution to
\begin{equation}\label{ScaledBol}
\begin{split}
&\frac{d}{d\tau} \int_V \varphi(w){f_\epsilon}(\tau,w, z)\ dw\\
&\qquad = \frac{1}{2\epsilon}\left \langle \iint_{V^2} (\varphi(w^{\prime})+\varphi(v^{\prime})-\varphi(v)-\varphi(w))   {f_\epsilon}(\tau, w,z){f_\epsilon}(\tau,v,z)dwdv \right\rangle. 
 \end{split}
\end{equation}

Hence, scaling the variance of the introduced random variables as $\sigma^2_\epsilon = \epsilon {\sigma}^2$ we can observe that for $\epsilon \ll1 $ the interactions become quasi-invariant since the differences $w^\prime-w$ and $v^\prime-v$ are small. Assuming now $\varphi$ sufficiently smooth and at least $\varphi \in \mathcal C_0^3(V)$ we can perform the following Taylor expansions
\[
\begin{split}
\varphi(w^\prime) - \varphi(w) &= (w^\prime-w) \partial_w\varphi(w) + \dfrac{1}{2}(w^\prime-w)^2 \partial_w^2\varphi(w) + \dfrac{1}{6}(w^\prime-w)^3 \partial_w^3\varphi(\bar w), \\
\varphi(v^\prime) - \varphi(v) &= (v^\prime-v)\partial_v \varphi(v) + \dfrac{1}{2}(v^\prime-v)^2 \partial_v^2\varphi(v) + \dfrac{1}{6}(v^\prime-v)^3 \partial_v^3\varphi(\bar v),
\end{split}
\]
where $\bar w \in (\min\{w,w^\prime\}, \max\{w,w^\prime\})$, $\bar v \in (\min\{v,v^\prime\}, \max\{v,v^\prime\})$. Plugging the above expansion in \eqref{ScaledBol} we have

\[
\begin{split}
&\dfrac{d}{d\tau} \int_V \varphi(w) {f_\epsilon}(\tau,w,z) dw \\
&\quad= \dfrac{1}{2\epsilon} \Big[ \left\langle\iint_{V^2} [(w^\prime-w) \partial_w\varphi(w) + (v^\prime-v) \partial_v\varphi(v)]{f_\epsilon}(\tau,w,z) {f_\epsilon}(\tau,v,z)\,dv\,dw\right\rangle  \\
&\quad +\dfrac{1}{2} \left\langle\iint_{V^2} [(w^\prime-w)^2 \partial_w^2\varphi(w)+ (v^\prime-v)^2\partial_v^2\varphi(v)]{f_\epsilon}(\tau,w,z) {f_\epsilon}(\tau,v,z)\,dv\,dw \right\rangle \\
&\quad   + R^\epsilon_\varphi({f_\epsilon},{f_\epsilon}), 
\end{split}
\]
where $R_\varphi^\epsilon({f_\epsilon},{f_\epsilon})$ is a reminder term with the following form 
\[
R_\varphi^\epsilon({f_\epsilon},{f_\epsilon})  = \dfrac{1}{6\epsilon} \left\langle \iint_{V^2}[ (w^\prime - w)^3 \partial_w^3\varphi(w) +  (v^\prime - v)^3 \partial_v^3\varphi(v) ]{f_\epsilon}(\tau,w,z){f_\epsilon}(\tau,v,z)\,dv\,dw  \right\rangle. 
\]
Thanks to the assumed smoothness we can argue that $\varphi$ and its derivatives are bounded in $V$. Furthermore, since $\eta_\epsilon$ has bounded moment of order three $\left \langle |\eta_\epsilon|^3 \right \rangle<+\infty$, we can observe that in the limit $\epsilon \rightarrow 0^+$ we have
\[
|R_\varphi^\epsilon({f_\epsilon},{f_\epsilon})| \rightarrow 0. 
\]
Therefore, in the limit $\epsilon \rightarrow 0^+$, it can be shown that $ f_\epsilon$ converges, up to subsequences, to a distribution function $\h = \h(\tau,w,z)$ which is weak solution to the following Fokker-Planck equation 

\begin{equation}\label{FP}
\begin{split}
&\partial_\tau {\h}(\tau,w,z)+\partial_w\left[\left( \int_V P(v,w,z) \h(\tau,v,z)dv \right)\h(\tau,w,z)\right]= \frac{{\sigma}^2}{2} \partial_w^2(D^2(w,z) {\h}(\tau, w,z)), 
\end{split} 
\end{equation}
where 
\[
P(v,w,z) = \dfrac{1}{2} \left[(p_1(z)+q_2(z)-2)w + (p_2(z)+q_1(z))v \right]
\]
provided the following boundary conditions are satisfied for all $z \in \Omega$
\begin{equation}
\label{eq:BC}
\begin{split}
-\left(\int_V P(v,w,z) \h(\tau,v,z)dv \right)\h(\tau,w,z) +\dfrac{\sigma^2}{2} \partial_w(D^2(w,z) {\h}(\tau, w,z)\Bigg|_{w \in \partial V} = 0 \\
D^2(w,z)  \h(\tau,w,z) \Bigg|_{w \in \partial V} = 0. 
\end{split}
\end{equation}

The asymptotic analysis of equation \eqref{FP} is considerable simpler compared to the original Boltzmann-type model and  convergence towards a unique equilibrium distribution can be obtained under suitable hypotheses, see \cite{carrillo1998exponential,furiolifokker,toscani1999entropy}. Therefore, we have obtained a surrogate model with reduced complexity whose large time behavior can be more easily studied.  In particular, the steady state distribution ${\h}_{\infty}(w,z)$ of \eqref{FP}  is determined by imposing
\begin{equation}\label{SteadyEQ}
\frac{{\sigma}^2}{2} \partial^2_w(D^2(w,z) {\h}_{\infty}(w,z))  =    \partial_w \left[\left(\int_V P(v,w,z) \h_\infty(v,z)dv \right)\h_\infty(w,z)\right]. 
\end{equation}
In fact, for many models the solution of the differential equation \eqref{SteadyEQ} is known analytically. In the next paragraph we present some relevant examples of socio-economic Boltzmann models which fit in the previously introduced general model. 

\subsection{Examples in socio-economic sciences}\label{examples} 

We shortly present two socio-economic models namely a rather general opinion formation model \cite{toscani2006kinetic} and the Cordier-Pareschi-Toscani (CPT) model \cite{cordier2009mesoscopic}. As we will see, the models are characterized by different steady states in their mean-field approximation.

We consider first a model for opinion formation where $V = [-1,1]$ and the binary interaction rules can be framed in \eqref{micro} in the symmetric case, i.e. putting in evidence the additional dependence of interaction by the opinion variable $p_1= q_2 = q(|v-w|,z)$, $p_2 = q_1= p(|v-w|,z)$, with additionally $q(|v-w|,z) = 1-p(|v-w|,z)$. Hence, the binary interaction scheme simplifies to 
\begin{align}\label{InterOpinion}
\begin{split}
&v^{\prime}= v + \epsilon  p(|w-v|,z) (w-v)+ D(v,z)   \eta_\epsilon, \\
&w^{\prime}= w+\epsilon p(|v-w|,z)( v-w) +D(w,z) \eta_\epsilon,
\end{split}
\end{align}
where the function $0\leq p(|v-w|,z) \leq 1$ weights the compromise tendency with respect to the relative opinion $|v-w|$. In the present case, in order to produce post-interaction opinions in the interval $[-1,1]$,  the random variable $\eta_\epsilon$ should be such that for all $z \in \Omega$ and $v,w \in V$ we have
\[
\begin{split}
-1-v - \epsilon p(|v-w|,z)(w-v)\le D(v,z)\eta_\epsilon \le 1-v - \epsilon p(|v-w|,z)(w-v). 
\end{split}\]
In particular for $D(w,z) = D(w)=  1-w^2$ and $p(|v-w|,z) = p(z)$ we obtain the following bound $|\eta|(1+|w|) \le  (1-\epsilon \max\{ p(z)\})$. In the quasi-invariant opinion limit described above and provided the boundary conditions \eqref{eq:BC} are satisfied we reduce to study the following model 
\[
\begin{split}
&\partial_t \h(t,w,z) + \partial_w \left[\left(\int_V p(|v-w|,z)(v-w)\h(t,v,z)dv  \right) \h(t,w,z) \right] \\
&\qquad= \dfrac{\sigma^2}{2} \partial_w^2 (D^2(w,z) \h(t,w,z)).
\end{split}
\]

In this case, in the interaction scheme \eqref{InterOpinion} we easily see that the mean opinion $m(z)$ is conserved in time and, if we consider also uncertainties in the initial distribution, the steady state distribution of the Fokker-Planck model reads

\begin{equation}\label{eq:steady_max}
{\h}_{\infty}(w,z)= C(z)\ (1+w)^{-2+\frac{p(z) m(z)}{2{\sigma} ^2}}  (1-w)^{-2-\frac{p(z) m(z)}{2 {\sigma}^2}} \ \exp\left\{-\frac{p(z) (1- m(z)w)}{\sigma^2\ (1-w^2)} \right\},
\end{equation}
where $C(z)$  is a normalization factor such that $\int_V \h_{\infty}(w,z)dw = 1$. Other form of equilibrium distribution can be determined, for example the choice $D(w) = \sqrt{1-w^2}$ produce a Beta-type steady state of the form 
\begin{equation}\label{eq:steady_beta}
\h_{\infty}(w,z)= 2^{1-\frac{2}{{\sigma^2}}       }\frac{1}{\textrm{B}\Big( \frac{1+m(z)}{{\sigma^2}} ,   \frac{1-m(z)}{{\sigma^2}}  \Big)}\ (1+w)^{\frac{1+m(z)}{{\sigma^2}}-1}  (1-w)^{\frac{1-m(z)}{{\sigma^2}}-1},
\end{equation}
where $\textrm{B}(\cdot,\cdot)$ is the Beta function. We refer to \cite{PTTZ,toscani2006kinetic} for a detailed discussion.

The second example that we mention has been presented in \cite{cordier2005kinetic} and models wealth exchanges between agents composing a simple economy. The wealth variable is here allowed to take values on the positive half line therefore $V=\R^+$. Again, the considered binary interactions can be framed in \eqref{micro} in the symmetric case, i.e. $p_1 = q_2 = q(z)$ and $p_2 = q_1 = \lambda(z)$ with $q(z) = 1-\lambda(z)$. Furthermore, we consider here $D(w,z) = w$. The uncertain parameter $\lambda(z) \in [0,1]$ determines the proportion of wealth that a single agents wants to invest, the quantity $1-\lambda(z)$ is the so-called saving propensity. Thus, the binary scheme reads for wealth exchanges reads
\begin{equation}
\label{InterWealth}
\begin{split}
&v^{\prime}= (1-\epsilon\lambda(z))v +\epsilon\lambda(z)w+v \eta_\epsilon\\
&w^{\prime}= (1-\epsilon \lambda(z)) w +\epsilon\lambda(z) v+w \eta_\epsilon,
\end{split}
\end{equation}
with $|\eta_\epsilon| \le (1-\epsilon)$.  Additionally, we assume that there is no uncertainty in the initial conditions of our Boltzmann-type model. 
 Hence, the corresponding mean field model provided  reads
\begin{align*}
\partial_t {\h}(t, w,z) +\lambda(z)\partial_w \left[(m_{\h}(z)-w) {\h}(t,w,z) \right]
= \frac{{\sigma}^2}{2} \partial_w^2 (w^2 {\h}(t,w,z)).
\end{align*}
We observe that the uncertain mean wealth $m_{\h}(z)$ is conserved in time. Hence, the equilibrium state can be computed and reads
\begin{equation}
\label{eq:steady_invgamma}
{\h}_{\infty}(w,z)= \frac{(\mu(z)-1)^{\mu(z)}}{\Gamma(w)\ w^{1+\mu(z)}}\exp\left( -\frac{(\mu(z)-1)m_{\h}(z)}{w}\right) ,
\end{equation}
where $\Gamma(\cdot)$ denotes the gamma function and $\mu(z):=1+\frac{2\ {\lambda}(z)}{{\sigma}^2}$. Notice that in this case the steady state exhibits tails with polynomial decay determined by the uncertain quantity $\mu(z)$. 

\section{MC-DSMC methods for uncertain Boltzmann equations}\label{prel}
In this section we introduce a MC-DSMC method to solve the uncertain Boltzmann equation where both the physical variables as well as the uncertain parameters are solved by Monte Carlo approximations. The realization of the method represents the starting point for the construction of our mean-field control variate strategy.

\subsection{Direct Simulation Monte Carlo method} The efficient computation of the highly non-linear Boltzmann model is a major task and has been tackled by several studies \cite{dimarco2014numerical,pareschi2001introduction,PZ2020}. In deterministic methods the multi-dimensional integral of the collision kernel needs to be approximated by a quadrature formula which suffers the so-called curse of dimensionality. Additionally, preservation of the main physical quantities is a true challenge at the discrete level that makes the scheme design model dependent.
On the other hand, Monte Carlo methods for kinetic equations naturally employ the microscopic dynamics to satisfy the physical constraints and are much less sensitive to the curse of dimensionality \cite{caflisch1998monte}. The most popular examples of Monte Carlo methods for the Boltzmann equation are the classical Direct Simulation Monte Carlo (DSMC) methods by Bird and Nanbu \cite{bird1970direct,pareschi2001introduction,nanbu1980direct}. The convergence of the methods to the solution of the Boltzmann equation has been rigorously proven in \cite{babovsky1989convergence, wagner1992convergence}. For a detailed introduction to DSMC solvers, especially for socio-economic Boltzmann type equations of the form  \eqref{BoltzmannModel}, we refer to \cite{pareschi2013interacting}.

In order to summarize the classical DSMC algorithms for simplicity we focus on the case without uncertainty. 
We are interested in the evolution of the density $f=f(w,t)$ solution of \eqref{ScaledBol} with initial condition $f(0,w)= f_0(w)$. We recall the symmetrized version of the simulation algorithm originally proposed by Nambu \cite{nanbu1980direct}, in a similar way one can consider Nanbu's algorithm \cite{bird1970direct}. Let us consider a time interval $[0,T]$, and let us
discretize it in $n_t$ intervals of size $\Delta t$. 

\begin{alg}[DSMC method]~
\begin{enumerate}
\item Compute the initial sample particles $ \{w_i^0, i=1,\ldots,N\} $,\\ 
      by sampling them from the initial density $f_0(w)$ 
\item   \begin{tabbing}
\= {\fP for} \= $n=0$  {\fP to} $n_t-1$ \\
          \>          \> given $\{w_i^n,i=1,\ldots,N\}$\\
          \>       \>   \ind \= \cir \= set $N_c = \IR(N\Delta t/2)$ \\
          \>       \>        \> \cir \> select $N_c$ pairs $(i,j)$ uniformly among all possible pairs, \\
          \>       \>        \>      \> - \= perform the collision between $i$ and $j$, and compute \\
          \>       \>        \>      \>      \> $w_i'$ and $w_j'$ according to the  collision law \\
          \>       \>        \>      \> - set $w_i^{n+1} = w_i'$, $w_j^{n+1}=w_j'$ \\
          \>       \>        \> \cir \> set $w_i^{n+1}=w_i^n$ for all the particles that have not been selected\\
         \> {\fP end for}
\end{tabbing}
\end{enumerate}
\label{al:dsmc1}
\end{alg}

Here, by $\IR(x)$ we denote the stochastic rounding of a
positive real number $x$
\[
   \IR(x) = \left\{\begin{array}{lll}
                     {\intg{x}} + 1 & \mbox{with probability} & x-{\intg{x}} \\
                     {\intg{x}}     & \mbox{with probability} & 1-x+{\intg{x}}
                   \end{array}
            \right.
\]
where $\intg{x}$ denotes the integer part of $x$.

The kinetic distribution, as well as its moments, is then recovered from the empirical density function
\begin{equation}
f_N(t,w)=\frac1{N}\sum_{i=1}^N \delta(w-w_i(t)),
\label{eq:emp}
\end{equation}
where $\delta(\cdot)$ is the the Dirac delta and $\{w_i(t), i = 1, \dots,N\}$ are the samples at time $t\ge 0$. %, or some suitable regularization of \eqref{eq:emp}.
For any test function $\varphi$, if we denote by 
\[
(\varphi, f)(t) = \int_V \varphi(w)f(t,w)dw, 
\]
we have
\[
(\varphi,f_N) (t) = \dfrac{1}{N} \sum_{i = 1}^N \varphi(w_i(t)). 
\]
Hence, by assuming that $\int_V f(t,w)dw= 1$ we have that $(\varphi,f)=\mathbb E_V[\phi]$, where $\mathbb E_V[\cdot]$ is the expectation of the observable quantity $\varphi$ with respect to the density $f$. In the sequel, we will also make use of the notation $\mathbb E[\cdot]$ to denote the expectation in the random space of uncertainties. We will implicitly assume that for multidimensional variables, as in the case of statistical samples, the expected values are done with respect to each variable.

Thanks to the central limit theorem the following result holds \cite{caflisch1998monte}:
\begin{lemma}\label{lem:1}
The root mean square error is such that for each $t\ge 0$
\begin{equation*}
\label{eq:th_e1}
\mathbb E_{V}\left[\left( (\varphi,f) - (\varphi,f_N) \right)^2\right]^{1/2}= \dfrac{\sigma_\varphi}{N^{1/2}}, 
\end{equation*}
where $\sigma^2_\varphi = \VV_V[\varphi]$ with
\begin{equation*}
\label{eq:sigma_varphi}
\VV_V[\varphi](t) = \int_V( \varphi(w)- (\varphi,f)(t))^2 f(t,w)dw.
\end{equation*}
\end{lemma}

If we are interested in the evaluation of the error produced in the approximation of structured quantities like the reconstruction of the distribution function we can operate as follows: we introduce a uniform grid in $V\subseteq \mathbb R$  where each cell has width  $\Delta w>0$ and we denote by $S_{\Delta w}\ge 0$ a smoothing function such that 
\[
\Delta w \int_{V} S_{\Delta w}(w)dw = 1.
\]
Then, we consider the approximation of the empirical density \eqref{eq:emp} obtained by
\begin{equation}
\label{eq:f_RN}
f_{N,\Delta w}(t,w) = \dfrac{1}{N} \sum_{i = 1}^N S_{\Delta w}(w-w_i(t)). 
\end{equation}
In the simplest case, we have $S_{\Delta w}(w)=\chi(|w|\leq \Delta w/2)/\Delta w$, where $\chi(\cdot)$ is the indicator function, that corresponds to the standard histogram reconstruction.

The numerical error of the reconstructed DSMC solution \eqref{eq:f_RN}, can be estimated from
\[
\begin{split}
\left\| f(t,\cdot)-f_{N,\Delta w}(t,\cdot)\right\|_{L^p(V,L^2(V))} & \leq 
\| f(t,\cdot)-f_{\Delta w}(t,\cdot)\|_{L^p(V)}\\
&+\| f_{\Delta w}(t,\cdot)-f_{N,\Delta w}(t,\cdot)\|_{L^p(V,L^2(V))}
\end{split}
\]
where
\[
f_{\Delta w}(t,w)=\int_{V} S_{\Delta w}(w-w_*)f(t,w_*)\,dw_*
\]
and we defined
\begin{equation}
\|g\|_{L^p(V,L^2(V))}=\|\mathbb E_{V}\left[g^2\right]^{1/2}\|_{L^p(V)}.
\label{eq:norm1}
\end{equation}
Now, the second term can be evaluated from Lemma 1 taking $\varphi(\cdot)=S_{\Delta w}(w-\cdot)$, $w\in V$, and gives
\[
\| f_{\Delta w}(t,\cdot)-f_{N,\Delta w}(t,\cdot)\|_{L^p(V,L^2(V))} = \frac{\|\sigma_S\|_{L^p(V)}}{N^{1/2}}
\]
where 
\begin{equation}
\sigma^2_S(t,w)=\VV_V[S_{\Delta w}(w-\cdot)](t).
\label{eq:sigmas}
\end{equation} 

Finally, the first term is bounded by
\[
\| f(t,\cdot)-f_{\Delta w}(t,\cdot)\|_{L^p(V)} \leq C_f (\Delta w)^q,
\]
accordingly to the accuracy used in the reconstruction. For example, $q=1$ in the case of simple histogram reconstruction, since
\[
\int_{V} \left|f(t,w)-\frac1{\Delta w}\int_{w-\Delta w/2}^{w+\Delta w/2} f(t,w_*)\,dw_*\right|^p dw \leq C \Delta w^p,
\]
where $C$ is a bound for the derivative of $f(t,w)$.
Therefore, in the general case, we have the following result.
\begin{theorem} The error introduced by the reconstruction function \eqref{eq:f_RN} satisfies
\begin{equation*}
\left\| f(t,\cdot)-f_{N,\Delta w}(t,\cdot)\right\|_{L^p(V,L^2(V))} \leq \frac{\|\sigma_S\|_{L^p(V)}}{N^{1/2}} + C_f (\Delta w)^q,
\end{equation*}
where $C_f$ depends on the $q$ derivative in $w$ of $f(t,w)$ and $\sigma_{S}^2$ is defined in \eqref{eq:sigmas}.
\label{th:1}
\end{theorem}
 \begin{remark}
It is worth to remark that in the above consistency estimates we neglected the errors due to time discretization. By adapting the arguments in \cite{Babo} to the present case it is possibile to show that Nanbu's scheme converges in law almost surely to the solution of the time
discretized space homogeneous Boltzmann equation \eqref{BoltzmannModel} in the deterministic case. We leave a detailed analysis of the convergence properties in connection with the specific form of the uncertain binary interaction \eqref{micro} to further researches. 
\end{remark}

\subsection{Quantity of Interest}\label{sect:QoI}
In order to analyze the uncertainty of our model \eqref{SymBoltzmannModel} one may not be only interested in the expected value and variance of the distribution function $f$ but rather in any quantity of interest (QoI) which can be computed from the distribution function. For that purpose we introduce the operator $\g[f](t,w,z)$ which is, in the simplest setting the identity, i.e. $\g[f] = f$. More generally, the QoI defined by $\g[f]$ is a functional of $f$, like, for example, the moments of the kinetic solution. 

Let $\Psi(z)$ be the probability distribution function of the uncertainty $z \in \Omega$, hence we define the expected value of the operator $\g[f](t,w,z)$ by
\begin{equation*}
\mathbb{E}[\g[f]] (t,w)=\int_{\Omega} \g[f](t,w,z)\Psi(z)dz,
\end{equation*}
whose variance is defined as
\begin{equation*}
\VV[\g[f]] (t,w)=\int_{\Omega} (\g[f](t,w,z)-\mathbb{E}[\g[f]](t,w))^2 \Psi(z)dz. 
\end{equation*}

In socio-economic applications, together with the moments of the distribution $f$ that are linked with observable quantities, we can define other operators characterizing the main features of the emerging distribution. The first example is given by the tail distribution
\begin{align*}
1-F(t,w,z):= \int_w^{\infty} f(t,v,z)dv.
\end{align*}
Other examples of interest are the Lorenz curve and the Gini coefficient in the case of wealth exchange models \cite{during2018kinetic,gastwirth1972estimation}. These quantities are frequently used measures in order to study the inequality of the wealth distribution. The Gini index can be computed from the Lorenz curve at the stationary state

\begin{equation}
L(F(w,z)) = \frac{\int_0^w v f_\infty(v,z)dv}{\int_0^\infty v f_\infty(v,z)dv}, 
\end{equation}
where 
\[
F(w,z) = \int_0^w f_\infty(v,z)dv, 
\]
as follows
$$
G_1(z)= 1-2 \int\limits_0^1 L(x,z)dx. 
$$
In the present setting the QoI defined by $\g[f]$ needs to be approximated by the DSMC solver of the Boltzmann equation. Therefore, given uncertain random samples $\{w_1(t,z), w_2(t,z),...,w_N(t,z)\}$, the corresponding empirical density $f_N$ is defined by \eqref{eq:emp}.

In particular, the particle approximation of the Lorenz curve at time $t$ reads
$$
L(F(t,w,z))\approx L_N(F_N(t,w,z)):= \frac{\sum_{w_i(t,z) \leq  w} w_i(t,z)}{\sum_{k=1}^N w_k(t,z)},
$$
with
$$
{F}_N(t,w,z):= \textrm{card}\{w_i(t,z) \leq w,\ i=1,...,N  \}.
$$

\subsection{Monte Carlo sampling method}
We shortly recall the standard MC sampling method for uncertainty quantification of a kinetic equation of type \eqref{BoltzmannModel}. Assume $f(t,w,z)$, $w \in V$, solution of a PDE with uncertainties only in the initial distribution $f_0(w,z)$, $z \in \Omega \subseteq \mathbb R^{d_z}$.  
The MC sampling method can be regarded as the simplest method for UQ and can be formulated as follows
\begin{alg}[MC-DSMC algorithm]
The algorithm can be compactly summarized in the following steps.
\begin{enumerate}
\item \textbf{Sampling:}  Sample $M$ independent identically distributed (i.i.d.) initial distribution $f^{k,0}=f_0(t,w,z_k),\ k=1,...,M$ from the random initial data $f_0$.
\item \textbf{Solving:} For each realization of $f^{k,0}$ the underlying kinetic equation \eqref{BoltzmannModel} is solved numerically by a DSMC solver for the kinetic equation.
We denote the solution at time $t^n$ by $f^{k,n}_N,\ k=1,...,M$ where $N$ is the sample size of the DSMC solver for the kinetic equation. 
\item \textbf{Reconstruction:} Estimate a statistical moment of the quantity of interest $\g[f^n]$ of the random solution field, e.g. for the expected value with the sample mean of the approximated solution $\g[f_N^{n}]$
\begin{align*}
\mathbb{E}[\g[f^n]] \approx E_M[\g[f_N^{n}]]:= \frac{1}{M} \sum\limits_{k=1}^M \g[f_N^{k,n}].
\end{align*}
\end{enumerate}
\end{alg}
The previous algorithm is straightforward to implement for the kinetic equation \eqref{BoltzmannModel}. Thus, the MC-DSMC method is non-intrusive and easily parallelizable since the ensemble averages are only computed as post-processing. 

The empirical kinetic distribution in presence of uncertainty is given by 
\[
f_N(t,w,z) = \dfrac{1}{N} \sum_{i=1}^N \delta(w-w_i(t,z)), 
\]
being $\{w_i(t,z), i = 1,\dots,N\}$ the samples of the particles at time $t\ge 0$ such that $w_i\in L^2(\Omega)$. For example, if $q[f]=(\varphi,f)$ we have
\[
q[f_N](z,t)=(\varphi,f_N)(z,t) = \dfrac{1}{N} \sum_{i=1}^N \varphi(w_i(z,t)),
\]
and the following result holds
\begin{lemma}
The root mean square error of the MC-DSMC method satisfies 
\[
\mathbb E\left[\mathbb E_{V}[ \left( \mathbb E[(\varphi,f)]- E_M[(\varphi,f_N)] \right)^2 ]\right]^{1/2}\le \frac{\nu_{(\varphi,f)}} {M^{1/2}} + \frac{\sigma_{\varphi,M}} {N^{1/2}}, 
\]
where $\nu^2_{(\varphi,f)} = \VV[(\varphi,f)]$
and $\sigma^2_{\varphi,M} = E_M[\sigma^2_{\varphi}]$ with $\sigma_\varphi^2=\VV_V[\varphi]$.
\label{lem:2}
\end{lemma}
\begin{proof}
The above estimate follows from 
\[
\begin{split}
\mathbb E\left[\mathbb E_{V}[ \left( \mathbb E[(\varphi,f)]- E_M[(\varphi,f_N)] \right)^2 ]\right]^{1/2}   \le & \,\mathbb E\Big[(\mathbb E[(\varphi,f)]  - E_M[(\varphi,f)])^2 \Big]^{1/2}  \\
 & + \mathbb E_{V}\Big[ \left( E_M[(\varphi,f)]  - E_M[(\varphi,f_N)] \right)^2\Big]^{1/2}  \\
  \leq   & \,\dfrac{\nu_{(\phi,f)}}{M^{1/2}}+ \dfrac{\sigma_{\varphi,M}}{N^{1/2}}, 
\end{split}
\]
where we used the fact that by Cauchy-Schwartz and Lemma \ref{lem:1} we have
\[
\begin{split}
\mathbb E_{V}\Big[ \left( E_M[(\varphi,f)]  - E_M[(\varphi,f_N)] \right)^2\Big]& \leq  \frac1{M}\sum_{k=1}^M \mathbb E_V[((\varphi,f)(z_k)-(\varphi,f_N)(z_k))^2]\\
& = \frac1{N}\left(\frac1{M}\sum_{k=1}^M \sigma^2_{\varphi}(z_k)\right).
\end{split}
\]
\end{proof}

Let us consider the reconstructed distribution with uncertainty
\begin{equation}
\label{eq:f_RNz}
f_{N,\Delta w}(t,w,z) = \dfrac{1}{N} \sum_{i = 1}^N S_{\Delta w}(w-w_i(t,z)), 
\end{equation}
and let us focus on the accuracy of the expectation of the solution $\mathbb E[f]$.
We can estimate

\[
\begin{split}
\left\| \mathbb E[f](t,\cdot)-E_M[f_{N,\Delta w}](t,\cdot)\right\|_{L^p(V,L^2(\Omega,L^2(V)))} & \leq 
\left\| \mathbb E[f](t,\cdot)-\mathbb E[f_{\Delta w}](t,\cdot)\right\|_{L^p(V)}\\
&+
\left\| \mathbb E[f_{\Delta w}](t,\cdot)-E_M[f_{N,\Delta w}](t,\cdot)\right\|_{L^p(V,L^2(\Omega))},
\end{split}
\]
where we defined 
\begin{equation}
\|g \|_{L^p(V,L^2(\Omega,L^2(V)))} = \| \mathbb E[\mathbb E_V[g^2]]^{1/2} \|_{L^p(V)}.
\label{eq:norm2}
\end{equation}
We can bound the first term as in the case without uncertainty
\[
\left\| \mathbb E[f](t,\cdot)-\mathbb E[f_{\Delta w}](t,\cdot)\right\|_{L^p(V)} \leq C_{\mathbb E[f]} (\Delta w)^q
\]
whereas the second term is bounded using Lemma \ref{lem:2} with $\phi(\cdot)=S_{\Delta w}(w-\cdot)$ to get
\[
\left\| \mathbb E[f_{\Delta w}](t,\cdot)-E_M[f_{N,\Delta w}](t,\cdot)\right\|_{L^p(V,L^2(\Omega,L^2(V))} \leq 
\frac{\|\nu_{(S,f)}\|_{L^p(V)}}{M^{1/2}}+\frac{\|\sigma_{S,M}\|_{L^p(V)}}{N^{1/2}},
\]
where
\begin{equation}
\nu^2_{(S,f)} = \VV[(S_{\Delta w}(w-\cdot),f)].
\label{eq:nus}
\end{equation}
As a consequence we have shown that
\begin{theorem}
The error introduced by the reconstruction function \eqref{eq:f_RNz} in the MC-DSMC method satisfies
\begin{equation*}
\label{eq:estim_2}
\begin{split}
\left\| \mathbb E[f](t,\cdot)-E_M[f_{N,\Delta w}](t,\cdot)\right\|_{L^p(V,L^2(\Omega,L^2(V)))}&\\
 \leq \frac{\|\nu_{(S,f)}\|_{L^p(V)}}{M^{1/2}}&+\frac{\|\sigma_{S,M}\|_{L^p(V)}}{N^{1/2}} + C_{\mathbb E[f]} (\Delta w)^q
\end{split}
\end{equation*}
where $\nu_{(S,f)}^2$ is defined in \eqref{eq:nus} and $\sigma^2_{S,M} = E_M[\sigma^2_{S}]$ with  $\sigma_S^2$ defined in \eqref{eq:sigmas}.
\label{th:2}
\end{theorem}

\begin{remark}
In the above consistency estimates we used the error norm defined in \eqref{eq:norm2}. A frequently used norm in uncertainty quantification based on Monte Carlo strategies is the expectation of the error defined as
\begin{equation}
\|g \|_{L^2(\Omega,L^p(V,L^2(V)))} = \mathbb E\left[\|\mathbb E_V[g^2]^{1/2}\|^2_{L^p(V)}\right]^{1/2}.
\label{eq:norm2b}
\end{equation}
The two norm for $p\neq 2$ differs and are related by Jensen's inequality
\[
\|g \|_{L^p(V,L^2(\Omega,L^2(V)))}\leq \|g \|_{L^2(\Omega,L^p(V,L^2(V)))}.
\]
We refer to \cite{dimarco2018multi} for a more detailed discussion.
\end{remark}

\section{Mean Field Control Variate DSMC methods}\label{MFCV}
In order to improve the accuracy of standard MC sampling methods, we introduce a novel class of mean field control variate  methods \cite{dimarco2018multi}. The key idea is to take advantage of the reduced cost of the mean field model which approximates the asymptotic behavior of the original Boltzmann model. This enables us to reduce the variance of the MC estimate. More precisely we consider two control variates approaches obtained by the the mean field approximation: the direct numerical solution of the mean field model and its corresponding steady state. 

The consistency of the mean field approximation  \eqref{FP} with the integro-differential equation \eqref{ScaledBol} has been previously discussed in Section \ref{model}. 
We denote the solution of the time scaled Boltzmann equation  \eqref{ScaledBol} by $f_{\epsilon}=f_{\epsilon}(t,w,z)$ and the solution of the corresponding mean field approximation \eqref{FP} by $\tilde{f}=\tilde{f}(t,w,z)$. The grazing collision regime of the time scaled kinetic equation \eqref{ScaledBol} corresponds to a small scaling parameter $\epsilon$.
Thus, up to an extraction of a subsequence, we indicate 
$$
\lim\limits_{\epsilon\to 0} f_\epsilon(t,w,z) = \h(t,w,z),
$$
being $f_\epsilon$ solution of \eqref{ScaledBol}. Therefore, also the equilibrium distributions are compatible, i.e.
$$
\lim\limits_{t\to\infty} \lim\limits_{\epsilon\to 0} f_\epsilon(t,w,z) = \h_{\infty}(w,z). 
$$
As we observed, the steady state $\h_{\infty}(w,z)$ of the Fokker-Planck model is analytically computable in several cases, whereas the steady state of the Boltzmann model \eqref{ScaledBol} is unknown.

The parameter dependent control variate method with $\lambda\in\R$ can be formulated by introducing the quantity 
 \begin{align}
  \g^{\lambda}[f_\epsilon]= \g[f_\epsilon]- \lambda (\g[\h]-\mathbb{E}[\g[\h]]).\label{CV}
\end{align}     
It is straightforward to observe that the expected value satisfies
$$
\mathbb{E}[\g^{\lambda}[f_\epsilon]]= \mathbb{E}[\g[\h]].
$$     
We can state the following
\begin{theorem}
The optimal value $\lambda^*$ which minimizes the variance of \eqref{CV} is given by
\begin{align}
\lambda^*:= \frac{\CC[\g[f_\epsilon],\g[\h]]}{\VV[\g[{\h}]]},
\label{eq:olambda}
\end{align}
where $\CC[\cdot,\cdot]$ denotes the covariance.
The corresponding variance of $q^{\lambda^*}[{f_\epsilon}]$ is then 
\begin{equation}\label{VarClev}
\VV[q^{\lambda^*}[f_\epsilon]]= \left(1-\rho^2_{\g[f_\epsilon], \g[\h]}\right)\ \VV[\g[f_\epsilon]],
\end{equation}
where
$$
\rho_{\g[f_\epsilon], \g[\h]}:=\frac{\CC[\g[f_\epsilon],\g[\h]]}{\sqrt{\VV[\g[f_\epsilon]]\  \VV[\g[\h]]}} \in (-1,1), 
$$
is the Pearson's correlation coefficient between $\g[f_\epsilon]$ and $\g[\h]$. 
In particular, we have 
$$
\lim\limits_{\epsilon \to 0}\frac{\CC[\g[f_\epsilon],\g[\h]]}{\VV[\g[\h]]}= 1,\quad \lim\limits_{\epsilon \to 0}\VV[\g^{\lambda^*}[f_\epsilon]] = 0.
$$
\label{th:lambda}
\end{theorem}
\begin{proof}
We aim to choose the coefficient $\lambda$ such that the variance $\VV[\g^{\lambda}[f_\epsilon]]$ of the control variate formulation is minimized. We have
\begin{align}\label{VarRed}
\VV[\g^{\lambda}[f_\epsilon]]=\VV[\g[f_\epsilon]]-2\ \lambda\ \CC[\g[f_\epsilon], \g[\h]]+ \lambda^2\ \VV[\g[{\h}]].
\end{align}
Then we can differentiate \eqref{VarRed} with respect to $\lambda$ 
to obtain that \eqref{eq:olambda} minimizes the variance. The second part of the theorem follows immediately since $\lim\limits_{\epsilon \to 0} f_\epsilon = \h$ holds in the quasi-invariant regime for $f_\epsilon$ solution to \eqref{ScaledBol}.
\end{proof}
Equation \eqref{VarClev} reveals that the mean field control variate approach may lead to a strong variance reduction provided that the correlation coefficient is close to one. 
                                                                                        
In the simplest setting the control variate formulation in \eqref{CV} can be modified using the steady state $\h_{\infty}(w,z)$ of the mean-field model \eqref{FP} 
 \begin{equation}
  \g^{\lambda}[f_\epsilon]= \g[f_\epsilon]- \lambda (\g[\h_{\infty}]-\mathbb{E}[\g[\h_{\infty}]]).\label{CVSteady}
\end{equation}  
For the mean field control variate steady state \eqref{CVSteady} similar results
holds in the large time limit
$$
\lim\limits_{t \to \infty}\lim\limits_{\epsilon \to 0} \frac{\CC[\g[{f}],\g[{\h}_{\infty}]]}{\VV[\g[{\h}_{\infty}]]}= 1,\quad \lim\limits_{t \to \infty}\lim\limits_{\epsilon \to 0}\VV[\g^{\lambda^*}[{f}]] = 0,
$$
where now
\[
\lambda^*:= \frac{\CC[\g[f_\epsilon],\g[\h^\infty]]}{\VV[\g[{\h^\infty}]]}.
\] 
 In practice it is only possible to compute the optimal $\lambda^*$ numerically. Furthermore, it is of paramount importance to be able to compute $\mathbb{E}[\g[{\h}]]$ or $\mathbb{E}[\g[{\h^\infty}]]$ exactly or with very small error in order to keep advantage of the control variate approach.

\subsection{A Mean-Field control variate estimator}
In a MC setting, we use $M$ realizations of our random variable $z$ to define the Mean Field Control Variate (MFCV) estimator
\begin{equation}
 \mathbb{E}[\g^{\lambda^*}[f_\epsilon]]\approx E_M[\g^{\lambda^*}[f_\epsilon]]= E_M[ \g[f_\epsilon]]-  \lambda^*_N (E_M[\g[\h]]-\mathbb E[\g[\h]]),
\end{equation}
where 
\begin{equation}
\lambda^*_N = \frac{Cov_M[\g[f_\epsilon],\g[\h]]}{Var_M[\g[\h]]},
\label{eq:lambdaN}
\end{equation}
and $\mathbb E[\g[\h]]$ denotes the exact value of the expectation of the considered QoI or its numerical approximation with negligible error. Furthermore, we introduced the following notations
\begin{align*}
& E_M[ \g[f_\epsilon]]:= \frac{1}{M}\sum_{k=1}^M \g[f_\epsilon^k],\quad E_M[ \g[\h]]:= \frac{1}{M}\sum_{k=1}^M \g[\h^k]\\
&Var_M[\g[\h]]:= \frac{1}{M-1}\sum\limits_{k=1}^M (\g[\h^k]- \mathbb{E}[\g[\h]])^2,\\
& Cov_M[\g[f_\epsilon], \g[\h]]:= \frac{1}{M-1}\sum\limits_{k=1}^M (\g[f^k_{\epsilon}]- E_M[\g[f_\epsilon]])\ (\g[\h^k]- \mathbb{E}[\g[\h]]), 
\end{align*}
with $f_\epsilon^k$ and $\h^k$ solutions of the Boltzmann-type model and of the mean-field model, respectively, relative to $k$th realization of the random variable $z$. 

The  MFCV algorithm based on a DSMC algorithm reads:
\begin{alg}[MFCV-DSMC algorithm]
The main steps of the MFCV-DSMC method for uncertainty in the initial data can be summarized as follows:
\begin{enumerate}
 \item \textbf{Sampling:}  Sample $M$ independent identically distributed (i.i.d.) samples of the initial distribution ${f}^{k,0} = f_0(w,z^k)$, $k=1,...,M$ from the random initial data ${f}_0(w,z)$.
 \item \textbf{Solving:} For each realization ${f}^{0,k}$, $k=1,...,M$
 \begin{itemize}
 \item Compute the control variate $\h^{k,n}$, $k=1,...,M$ at time $t^n$ solving with a suitable deterministic method the mean field model \eqref{FP} (or using the steady state $\h_{\infty}$ ) and compute $\mathbb E[\g[\h^n]]$ (or $\mathbb E[\g[\h_\infty]]$) with negligible error. 
 \item Solve the kinetic equation \eqref{ScaledBol} by a DSMC solver with sample size $N$. We denote the solution at time $t^n$ by ${{f}}^{k,n}_{\epsilon,N},\ k=1,...,M$. 
 \end{itemize}
 \item \textbf{Estimating:}
  \begin{itemize}
 \item Estimate the optimal value of $\lambda^{*}$ at time $t^n$ by \eqref{eq:lambdaN} and denote it as $ \lambda^{*,n}_M$.
 \item Compute the the expectation of any quantity of interest $\g[f_{\epsilon,N}]$ of the random solution field with the mean-field control estimator
 $$
 E^{\lambda_*}_M[g[f_{\epsilon,N}]]= E_M[ \g[f^{n}_{\epsilon,N}]]-  \lambda^{*,n}_M (E_M[\g[\h^{n}]]-\mathbb E[\g[\h^n]]).
 $$
 \end{itemize}
\end{enumerate}
\end{alg}

The error bound of the MFCV-DSMC method may improve significantly  in comparison to the standard MC method thanks to the relation  \eqref{VarClev} once 
 $\rho_{\g[f_\epsilon], \g[\h]}\approx 1$ holds. Concerning the evaluation of moments, by ignoring the error term due to the approximation of $\lambda_*$, we have the following
 \begin{lemma}
The root mean square error of the MFCV-DSMC method satisfies 
\[
\mathbb E\left[\mathbb E_{V}\left[ \left( \mathbb E[(\varphi,f_\epsilon)]- E^{\lambda_*}_M[(\varphi,f_{\epsilon,N})] \right)^2 \right]\right]^{1/2}\le \left(1-\rho^2_{(\varphi,f_\epsilon), (\varphi,\h)}\right)^{1/2}\frac{\nu_{(\phi,f_\epsilon)}} {M^{1/2}} + \frac{\sigma_{\varphi,M}} {N^{1/2}}, 
\]
where $\nu^2_{(\varphi,f_\epsilon)} = \VV[(\varphi,f_\epsilon)]$
and $\sigma^2_{\varphi,M} = E_M[\sigma^2_{\varphi}]$ with $\sigma_\varphi^2=\VV_V[\varphi]$.
\label{lem:3}
\end{lemma}
\begin{proof}
Let us observe that
\[
\mathbb E[(\varphi,f_\epsilon)]=\mathbb E[(\varphi,f^{\lambda_*}_\epsilon)],\qquad E^{\lambda_*}_M[(\varphi,f_{\epsilon,_N})=E_M[(\varphi,f^{\lambda_*}_{\epsilon,N})]
\]
and then, using Theorem \ref{th:lambda}, we get 
\[
\nu^2_{(\varphi,f_\epsilon^{\lambda_*})} = \VV[(\varphi,f_\epsilon^{\lambda_*})] = \left(1-\rho^2_{(\varphi,f_\epsilon), (\varphi,\h)}\right)\ \VV[(\varphi,f_\epsilon)].
\]
The conclusion follows from Lemma \ref{lem:2} together with the identity
\[
(\varphi,f_\epsilon^{\lambda_*})-(\varphi,f_{\epsilon,N}^{\lambda_*})=(\varphi,f_\epsilon)-(\varphi,f_{\epsilon,N}).
\]
\end{proof}
In the case of the reconstruction function \eqref{eq:f_RNz}, we have
\[
\begin{split}
\left\| \mathbb E[f_\epsilon](t,\cdot)-E^{\lambda_*}_M[f_{\epsilon,N,\Delta w}](t,\cdot)\right\|_{L^p(V,L^2(\Omega,L^2(V)))} &  \\
&\hskip -1.3cm \leq \left\| \mathbb E[f_\epsilon](t,\cdot)-\mathbb E[f_{\epsilon,\Delta w}](t,\cdot)\right\|_{L^p(V)}\\
&\hskip -1.3cm +\left\| \mathbb E[f_{\epsilon,\Delta w}](t,\cdot)-E^{\lambda_*}_M[f_{\epsilon,N,\Delta w}](t,\cdot)\right\|_{L^p(V,L^2(\Omega,L^2(V)))},
\end{split}
\]
where the first term is bounded as in Theorem \ref{th:2} and the second term can be bounded using Lemma \ref{lem:3} with $\phi(\cdot)=S_{\Delta w}(w-\cdot)$. Thus we proved the following result.
 \begin{theorem}
The error introduced by the reconstruction function \eqref{eq:f_RNz} in the MFCV-DSMC method satisfies
\begin{equation}
\label{eq:estim_3}
\begin{split}
\left\| \mathbb E[f_\epsilon](t,\cdot)-E^{\lambda_*}_M[f_{\epsilon,N,\Delta w}](t,\cdot)\right\|_{L^p(V,L^2(\Omega,L^2(V)))}&\\
 \leq 
 \frac{\left\|\left(1-\rho^2_{(S,f_\epsilon), (S,\h)}\right)^{1/2}\nu_{(S,f_\epsilon)}\right\|_{L^p(V)}} {M^{1/2}} + \frac{\|\sigma_{S,M}\|_{L^p(V)}} {N^{1/2}} + C_{\mathbb E[f]} (\Delta w)^q
\end{split}
\end{equation}
where and $\nu_{(S,f_\epsilon)}^2$ is defined as in \eqref{eq:nus} and $\sigma^2_{S,M}$ is defined in Theorem \ref{th:2}.
\label{th:4}
\end{theorem}
As a consequence when the solution of the full model is close to the solution of the control variate the statistical error due to the uncertainty vanishes. This justifies the use of a larger number of samples in the state space in agreement with the reconstruction used in order to balance the last two error terms in \eqref{eq:estim_3}.

\section{Numerical Examples}
In this section, we present several numerical examples of the mean field control variate method with application to different kinetic models in socio-economic sciences. 
The presented test cases consider uncertain initial data and uncertain interaction coefficients as well. We start presenting relevant tests on the steady state control variate approach that, we remark, has no impact on the simulation costs, since the expected value of steady states can be computed offline with arbitrary accuracy.  Subsequently, we consider the general mean field control variate approach where the mean-field model is solved using a second order structure preserving finite difference method \cite{pareschi2018structure}. In this case, we cannot perform the computation of $\mathbb E[\g[\cdot]]$ offline. Hence, the computational costs play a relevant role and we will discuss the connection between simulation costs and performance. 

\subsection{Mean Field Steady State Control Variate (MFCV-S)}

We consider in this section the case where the steady state $\h_\infty(w,z)$ of the mean-field model is analytically known and is used as control variate. We will refer to this method, where the control variate term is evaluated off line, as  Mean Field Steady State Control Variate (MFCV-S). If not otherwise indicated the quantity of interest is here the distribution itself $\g[f_\epsilon] = f_\epsilon$. 
The reconstruction is performed using standard first order histogram approximation with $N_Z=100$ grid points. In order to compute with negligible error $\mathbb E[\h_\infty]$ we adopt a stochastic collocation approach with $20$ collocation nodes.

\subsubsection*{Test 1: Opinion model with uncertainty}\label{sect:opinion}

\begin{figure}
\begin{center}
\includegraphics[scale = 0.35]{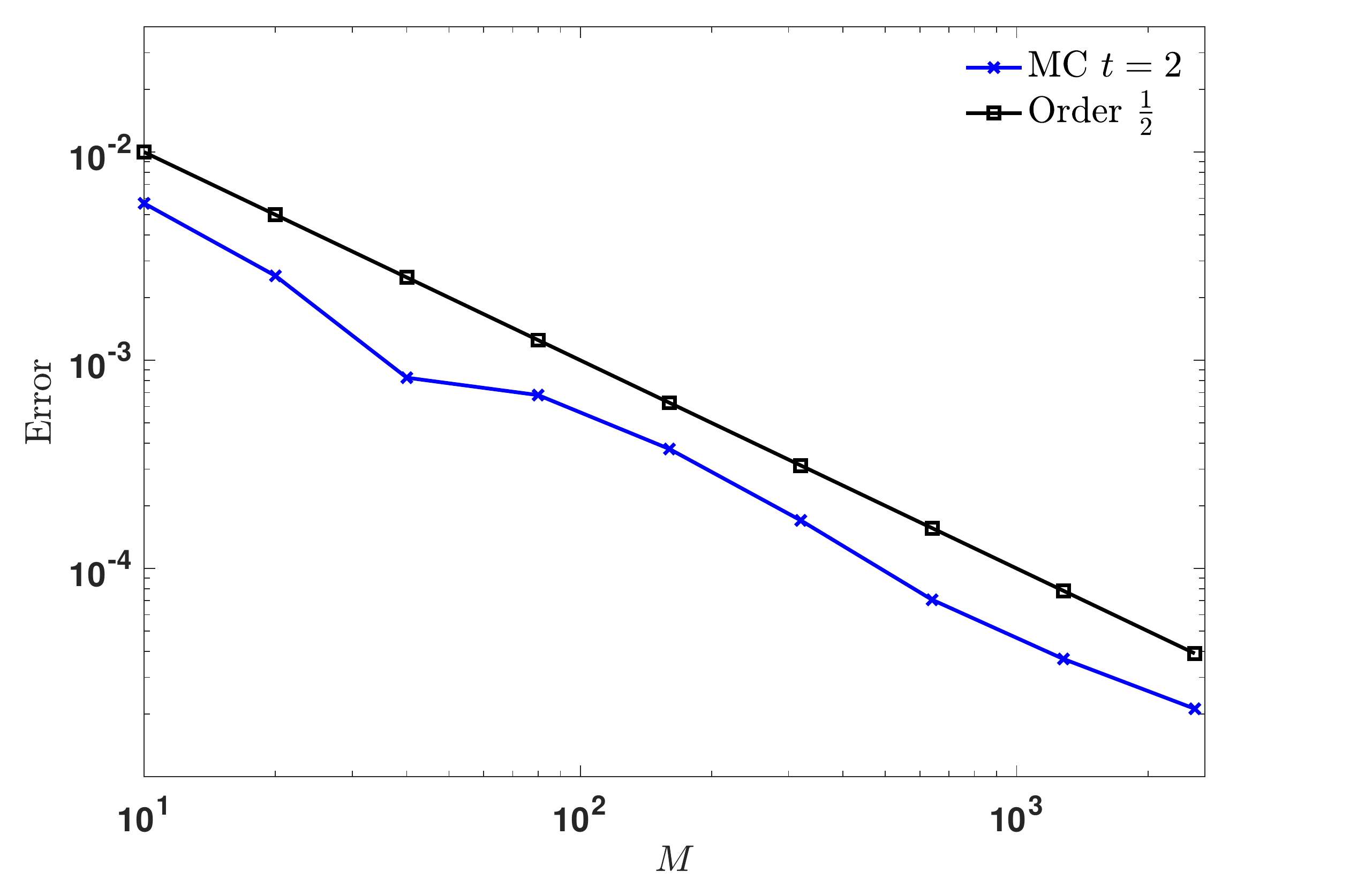}
\caption{\textbf{Test 1}. Convergence of the expected value of the solution $\mathbb E[f_\epsilon]$, in the case of the Boltzmann model for opinion formation \eqref{eqn:A} with uncertain initial data, for the classical MC sampling method at time $t=2$. The solutions are averaged over $50$ runs to reduce statistical fluctuations. The DSMC solver has been implemented with $N=2\times 10^4$ samples and $\epsilon=10^{-1}$. }\label{MCFig}
\end{center}
\end{figure}

We consider the kinetic model for opinion formation resulting from binary interactions \eqref{InterOpinion} with uncertainties present on the initial distribution or on the interaction strength. 
Let us first assume that the uncertainty is present in the initial distribution of the problem. More precisely, we assume that $z\sim \mathcal{U}([0, 1])$ and the initial distribution $f_0(w,z)$ is given by
\begin{equation}
f_0(w,z) = 
\begin{cases}
1 & w \in \left[\dfrac14 (z-2) , \dfrac14 (z+2)\right] \\
0 & \textrm{otherwise}. 
\end{cases}
\label{eq:t1}
\end{equation} 
Hence, the first moment $m(z) = {z}/{4}$  is conserved in time. We also assume 
\begin{equation}
\label{eqn:A}
p(|v-w|,z) = 1,\qquad D(v)= \sqrt{1-v^2},
\tag{A}
\end{equation}
and thus we obtain a steady state of the Fokker-Planck equation of the form $\h_\infty(w,z)$ given by \eqref{eq:steady_beta}.  
The Boltzmann equation is solved with $N=2\times 10^4$ particles up to $t=5$ where the solution approaches the steady state. 

In Figure \ref{MCFig},  we report the $L^2$ error of $\mathbb E[f_\epsilon]$ at $t=5$ for different number of samples $M$ from which we can clearly deduce the MC convergence rate of $\frac12$. The convergence towards $\mathbb E[\h_\infty]$ is given in the top row of Figure \ref{Test1Dist} for different  $\epsilon = 10^{-1},10^{-2}$ and different number of samples $M = 20$ (left plot) and $M=1280$ (right plot). The bottom row of Figure  \ref{Test1Dist} shows that using the MFCV-S approach, in comparison to the classical MC method, for $M=20$ we obtain already a good fit to the mean field steady state for $\epsilon=10^{-2}$. 

\begin{figure}
\begin{center}
\includegraphics[scale = 0.28]{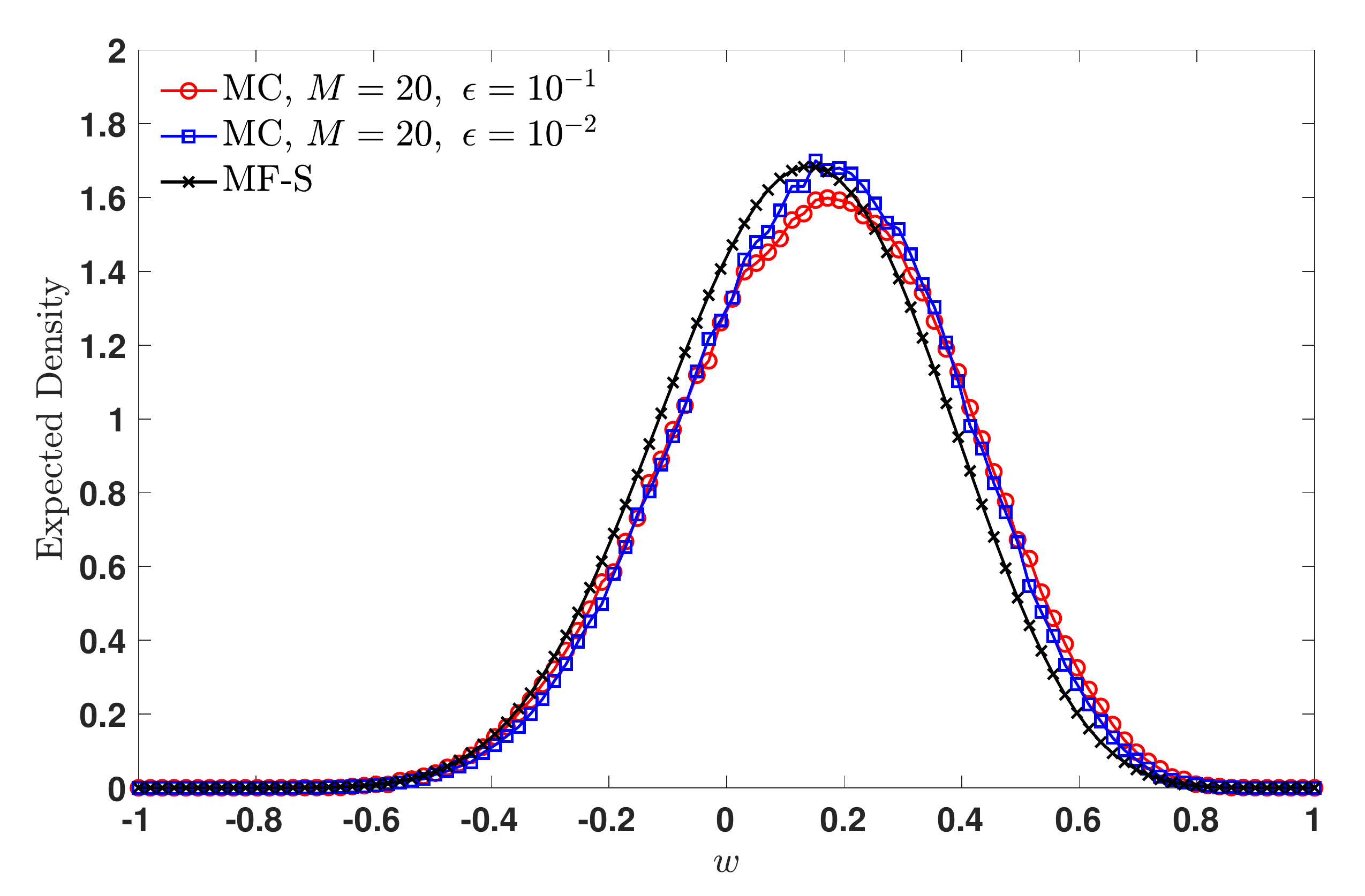}
\includegraphics[scale = 0.28]{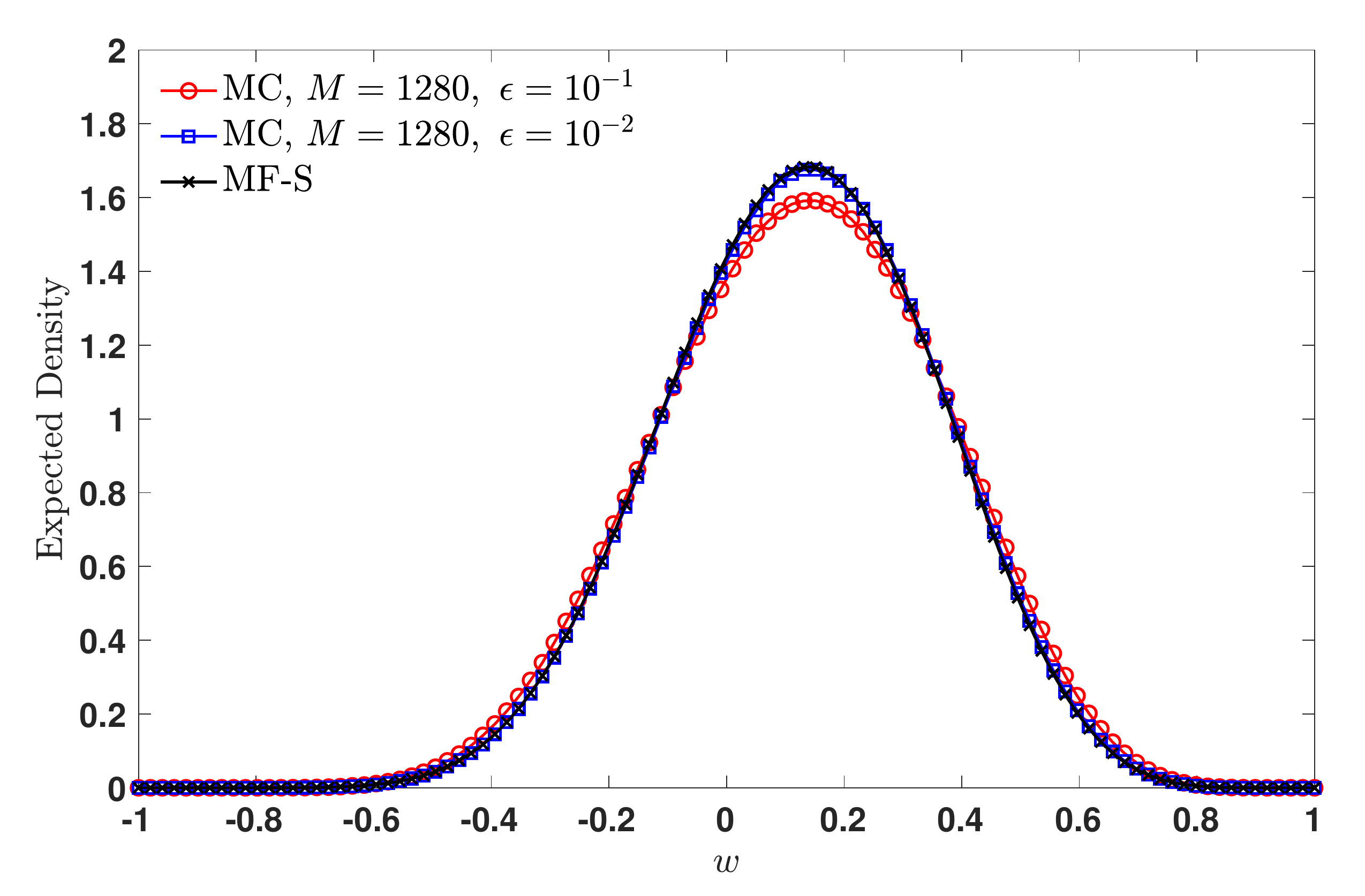} \\
\includegraphics[scale = 0.28]{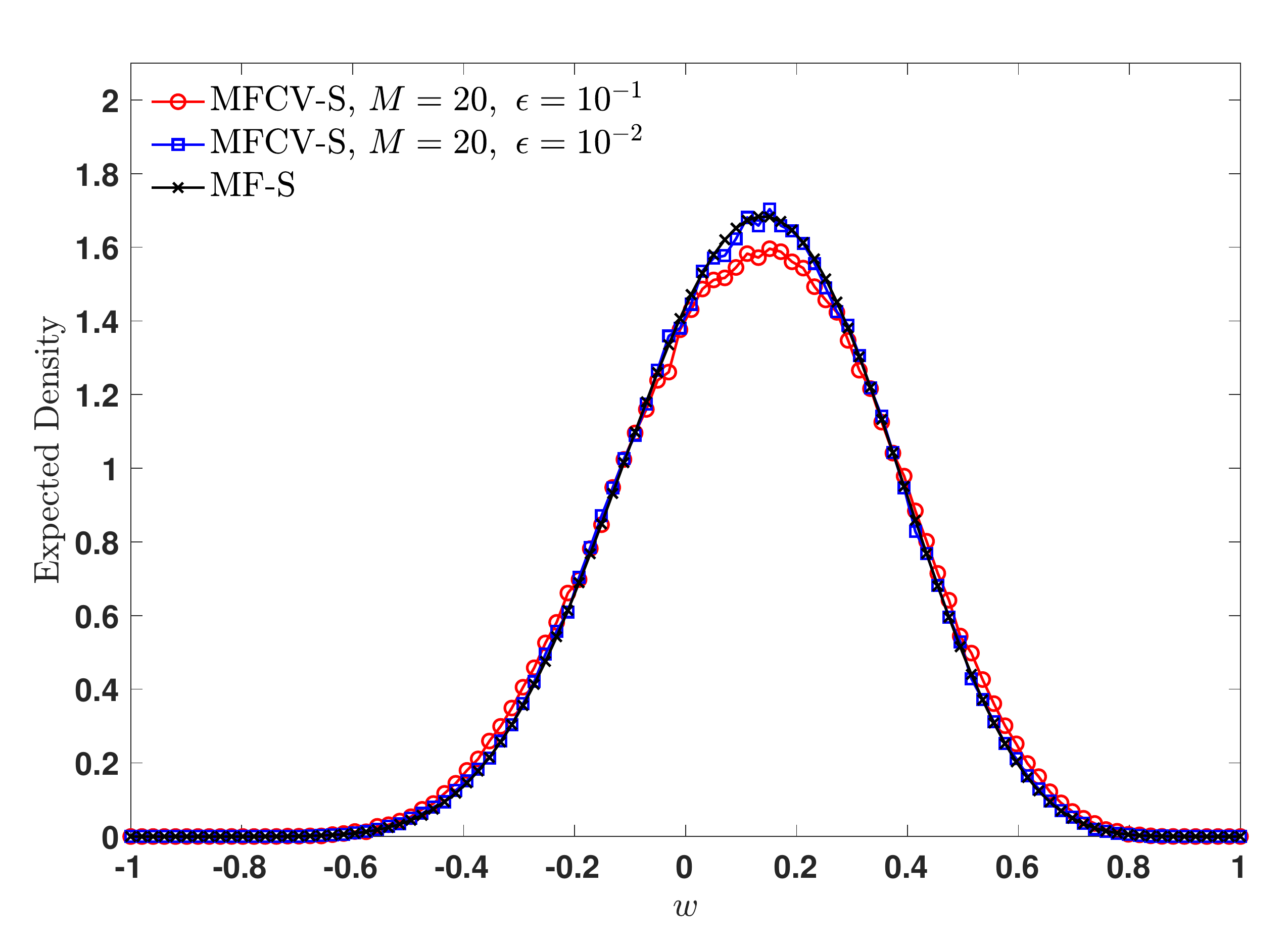}
\includegraphics[scale = 0.28]{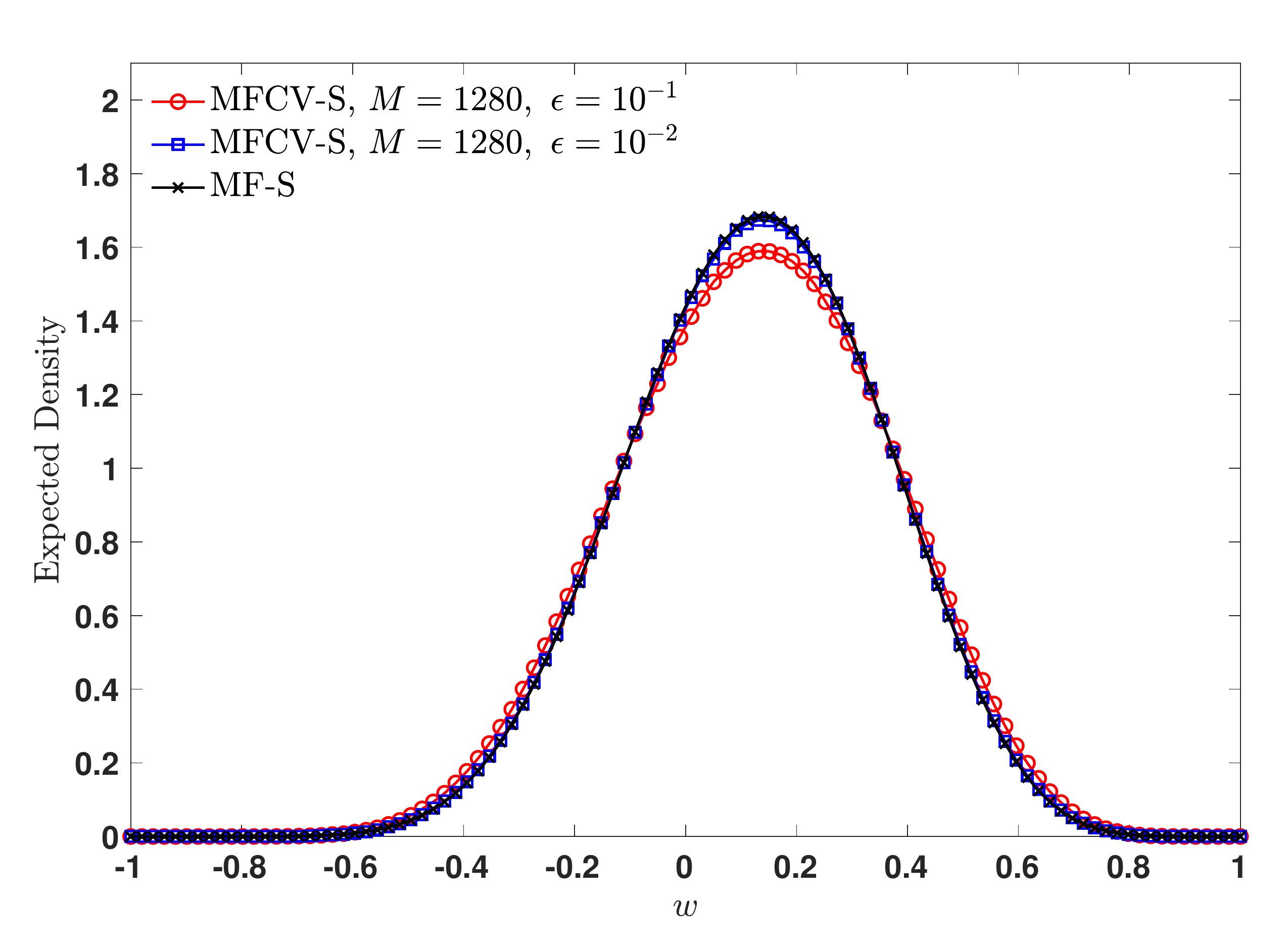}
\caption{\textbf{Test 1}. Approximation of $\mathbb E[f_\epsilon]$, in the case of the Boltzmann model for opinion formation \eqref{eqn:A} with uncertain initial data. The DSMC solver has been implemented with $N=2\times 10^4$. We considered different scalings $\epsilon = 10^{-1}, 10^{-2}$ and two different number of samples $M = 20$ (left) and $M = 1280$ (right).  The cross markers represent $\mathbb E[\h_\infty]$. (Top) standard MC sampling; (Bottom) MFCV-S method. }\label{Test1Dist}
\end{center}
\end{figure}

In Figure \ref{ErrorSamples} (left) we report the $L^2$ error of the expected density obtained by the MC and MFCV-S methods for increasing number of samples at time $t = 5$. We obtain an improvement in accuracy for the MFCV-S between one and two orders of magnitude using the same number of samples. 
\begin{figure}
\begin{center}
\includegraphics[scale = 0.28]{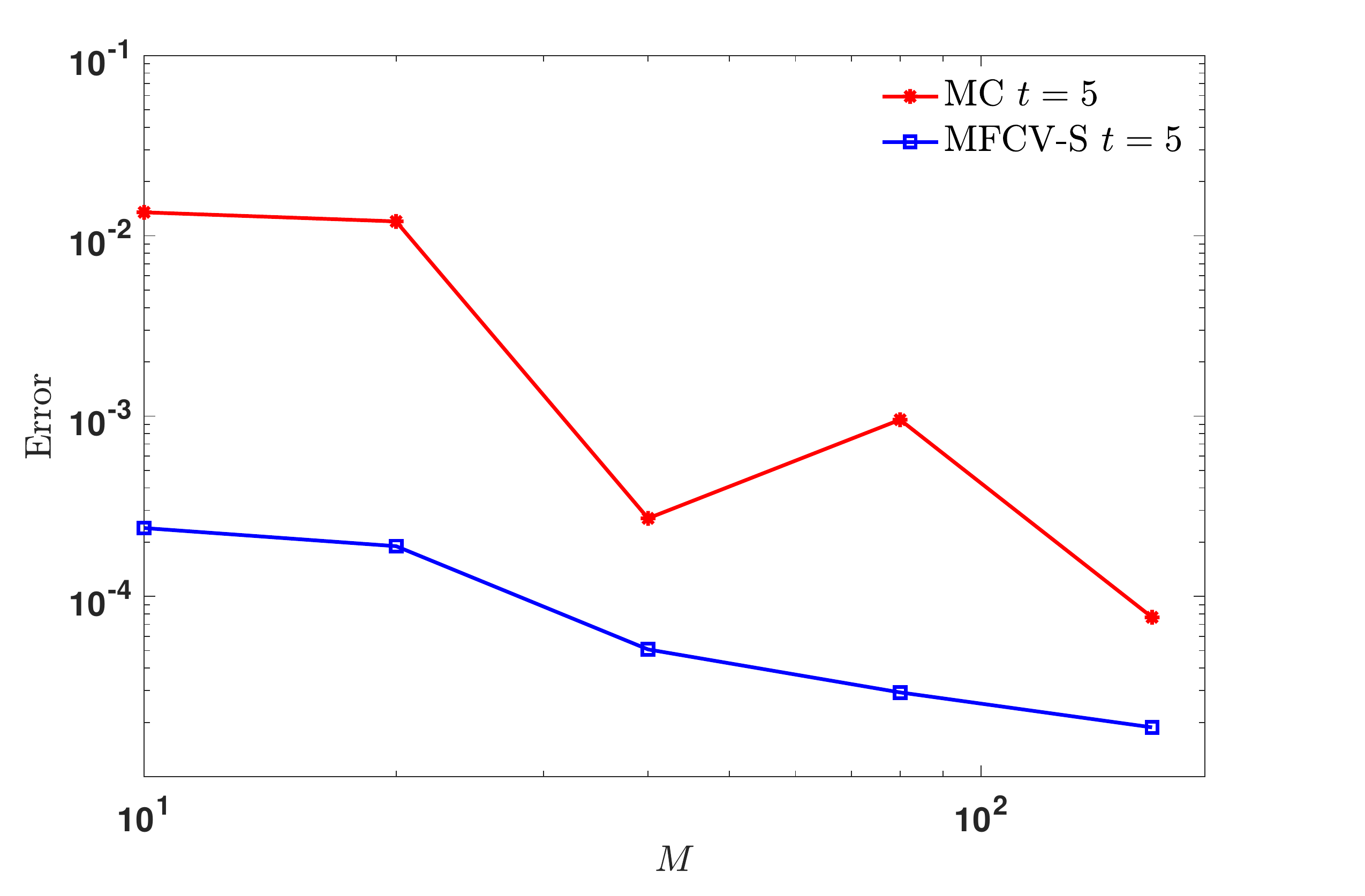}
\includegraphics[scale = 0.28]{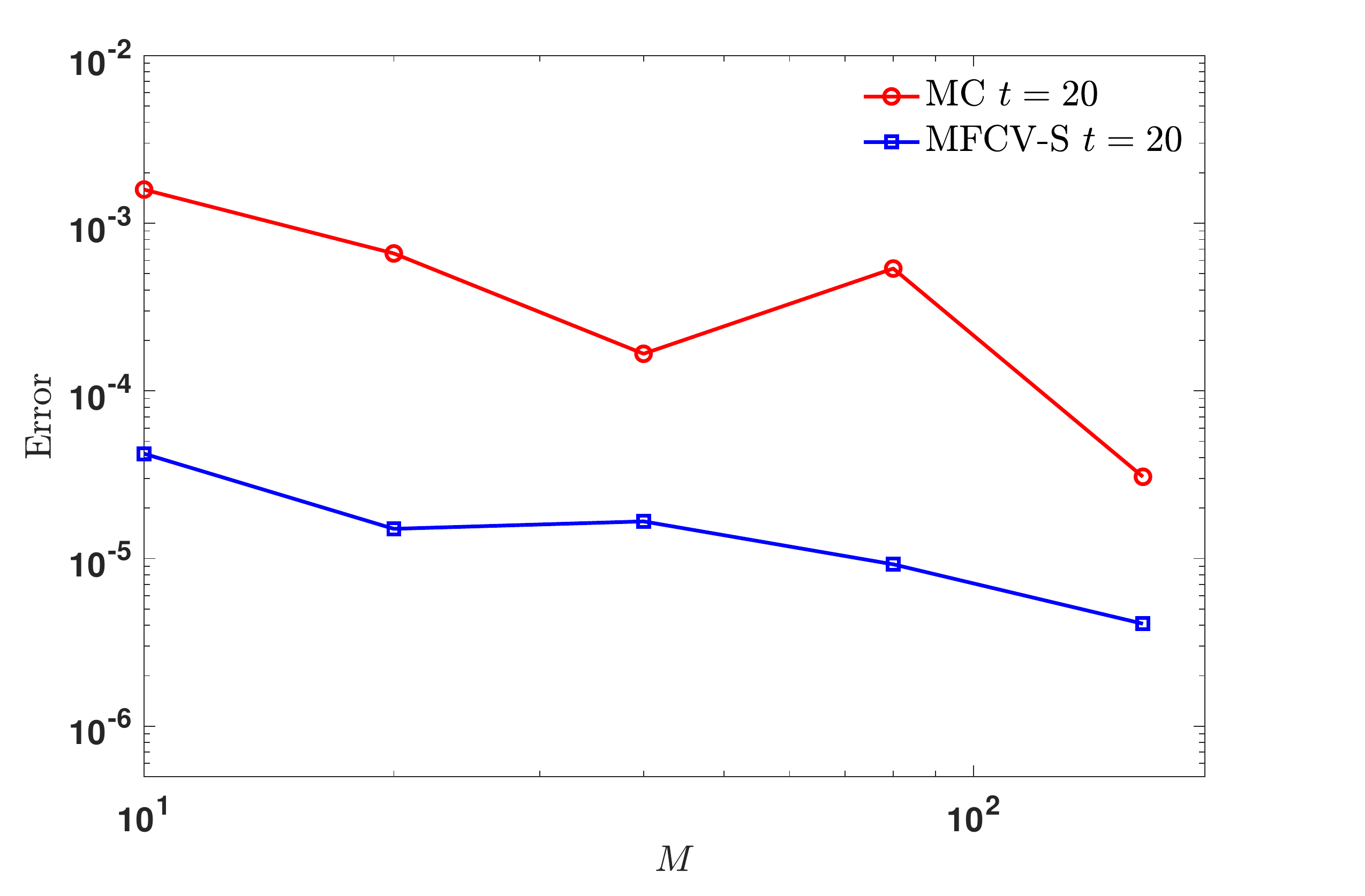}
\caption{\textbf{Test 1}. Error of the MFCV-S estimate and classical MC method in the case of the Boltzmann model for opinion formation for increasing samples $M$. We considered $N=2\times 10^4$ in the DSMC solver. (Left) Solution at $t=5$ of case \eqref{eqn:A} with uncertain initial data; (Right) Solution at $t=20$ of case \eqref{eqn:B} with uncertain interaction parameters.}\label{ErrorSamples}
\end{center}
\end{figure}

As a further test for opinion dynamics we consider the kinetic model resulting from binary interactions \eqref{InterOpinion} with 
\begin{equation}
\label{eqn:B}
p(|w-v|,z)= \dfrac{3}{4}+\dfrac{z}{4},\qquad z\sim\mathcal{U}([-1,1]),\qquad D(w) = {1-w^2},
\tag{B}
\end{equation}
so that the resulting steady state of the Fokker-Planck model is the Maxwellian-like distribution \eqref{eq:steady_max}. The initial data in this case is \eqref{eq:t1} in the deterministic setting $z=0$ and the final time is $t=20$.

In Figure \ref{ErrorSamples} (right) we report the $L^2$ error of expected probability distribution function computed by the MC and MFCV-S method at the final time for different number of samples. As in case \eqref{eqn:A} we obtain an improvement between one and two orders of accuracy for the MFCV-S method compared to the classical MC method.

\subsubsection*{Test 2: Wealth model with uncertainty}

We study two test cases of the CPT model defined by the binary interaction scheme \eqref{InterWealth}. First, we consider uncertainty in the initial condition and secondly in the saving propensity. We will consider as computational domain the interval $[0,10]$.

Let us first assume that the uncertainty is present in the initial distribution of the problem. In details, we consider $z \in \mathcal U([0,1])$ and the initial  distribution $f_0(w,z)$ defined by 
\begin{equation*}
f_0(w,z) = 
\begin{cases}
\dfrac{1}{2} & w \in \left[\dfrac{z}{5},2+\dfrac{z}{5} \right] \\
0 & \textrm{otherwise}.
\end{cases}
\end{equation*}  
Furthermore, we consider 
\begin{equation}
\lambda(z) = 1,\qquad D(w) = w,
\tag{A}
\label{eqn:A2}
\end{equation}
so that the large time behavior of the Fokker-Planck model $\h_\infty(w,z)$ is given by \eqref{eq:steady_invgamma} with $m(z) = 1+ \dfrac{z}{5}$.

In Figure \ref{QoITest2}, we show the MC and MFCV-S approximations of $\mathbb E[f_\epsilon]$ and of the expected value of the Lorenz curve, see Section \ref{sect:QoI}, at the final time $t=30$ for $\epsilon=10^{-2}$. We can observe a better agreement of the solution of the MFCV-S method with the 
expected steady state solution of the mean field model than the solution obtained by the MC solver. 

In Figure \ref{CVTest2} we compare the expected QoI computed by the MC and MVCV-S solver for different number of samples $M$ plotted  against time. We can observe that the MC method needs at least $8$ times more samples in order to reach the same accuracy than the MFCV-S method with $10$ samples. 

\begin{figure}
\begin{center}
\includegraphics[width=0.5 \textwidth]{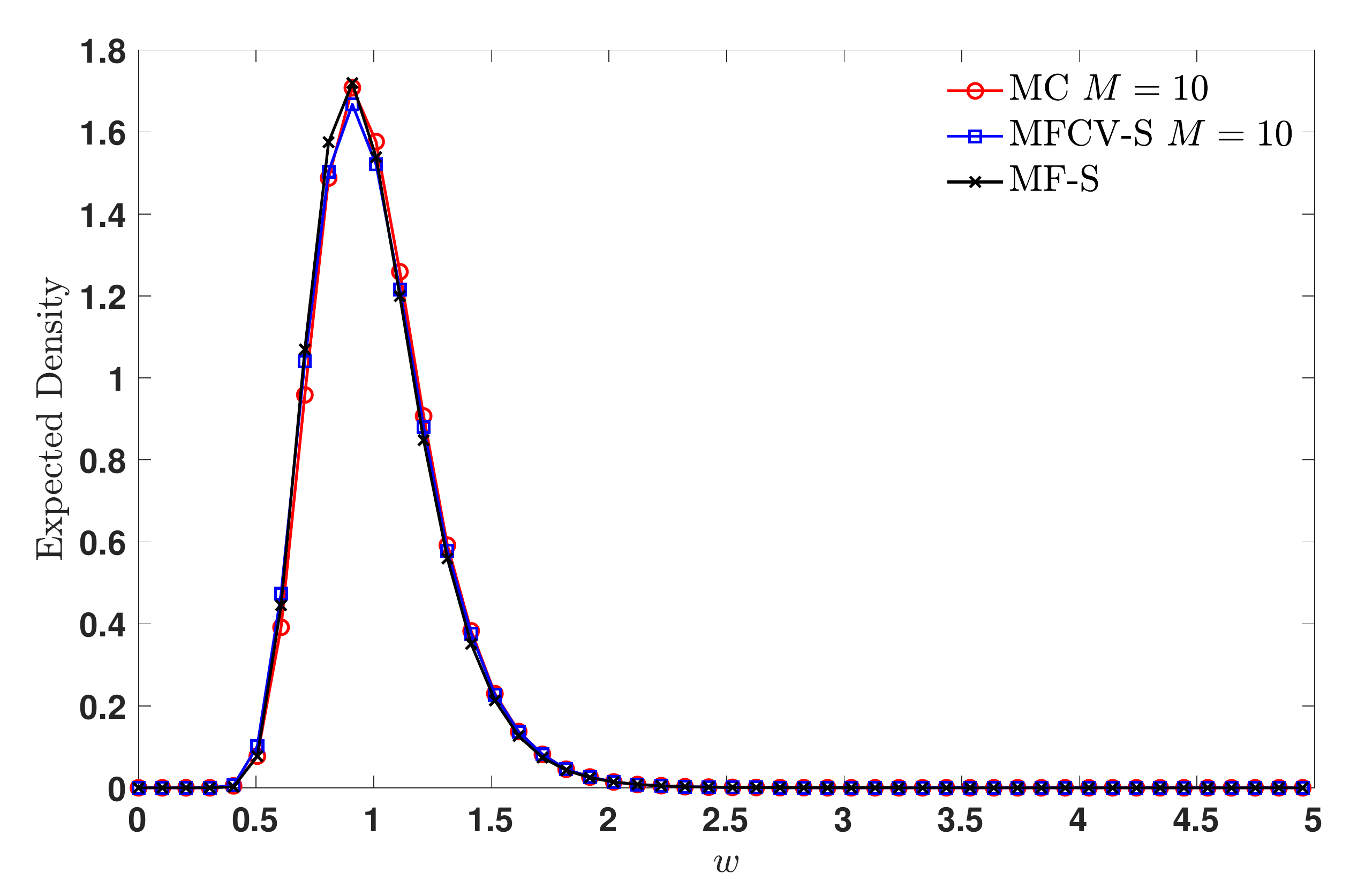}\hfill
\includegraphics[width=0.5 \textwidth]{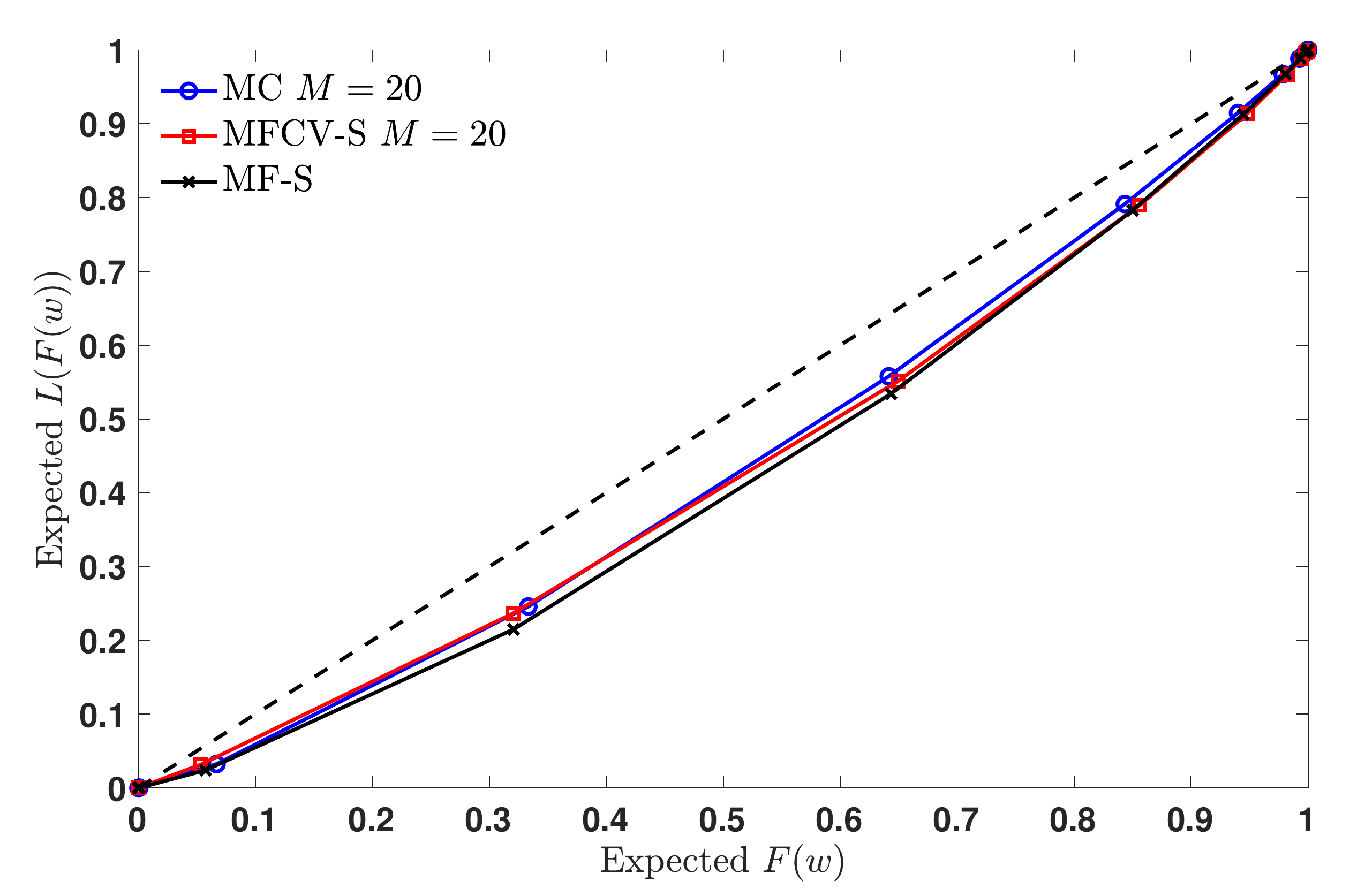}
\caption{\textbf{Test 2}. Expected distribution $\mathbb E[f_\epsilon]$ and expected Lorentz curve for the MC, MFCV-S and the mean field model for problem \eqref{eqn:A2}. The DSMC solver of the Boltzmann model used $N=2\times 10^4$ particles. 
Left: Expected distribution function $\mathbb E[f]$ computed with $M=10$ samples for the MC and MFCV-S method. Right: Expected Lorenz curve computed with $M=20$ samples for the MC and MFCV-S method.}\label{QoITest2}
\end{center}
\end{figure}

\begin{figure}
\begin{center}
\includegraphics[width=0.5 \textwidth]{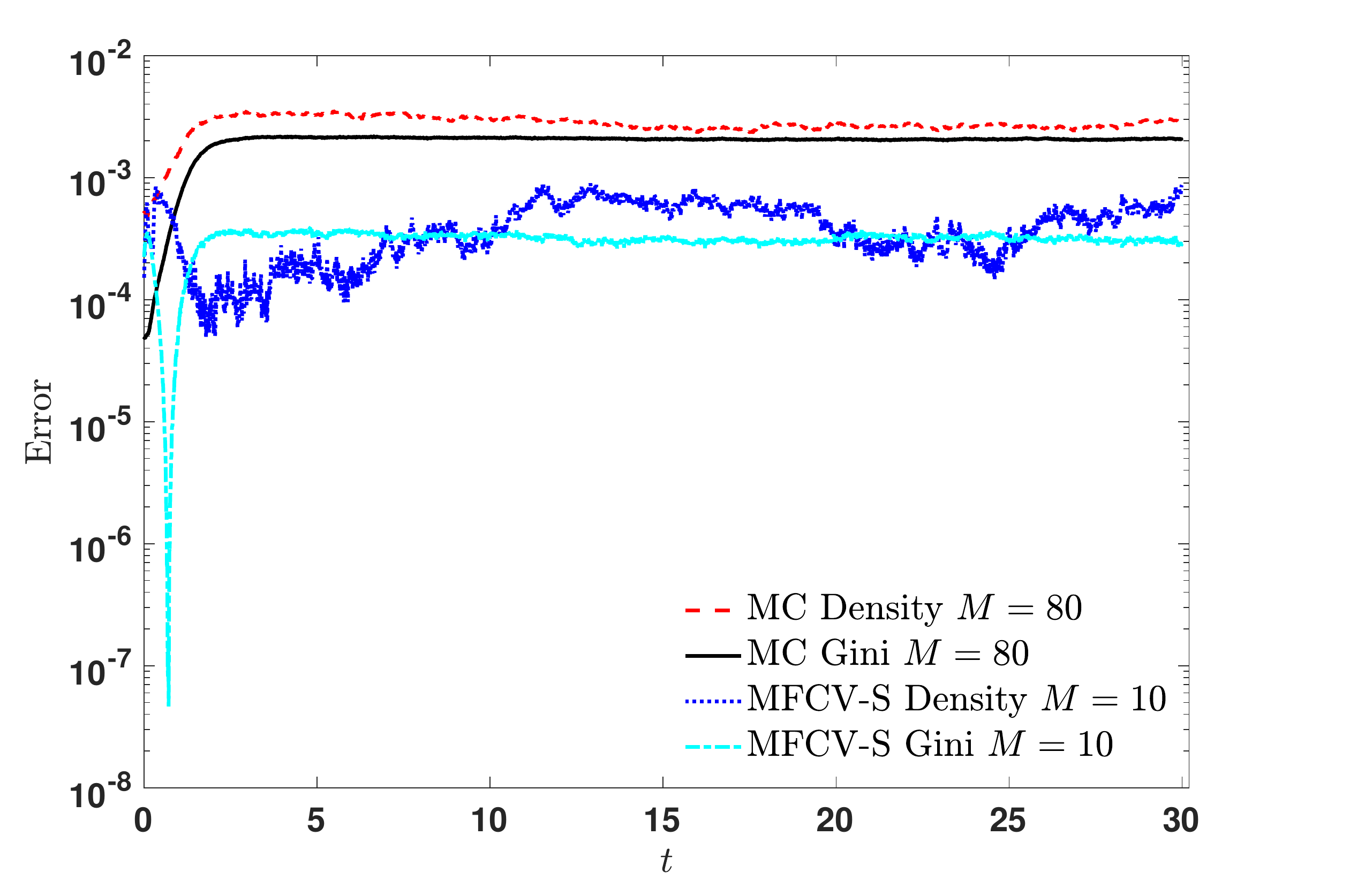}\hfill
\includegraphics[width=0.5 \textwidth]{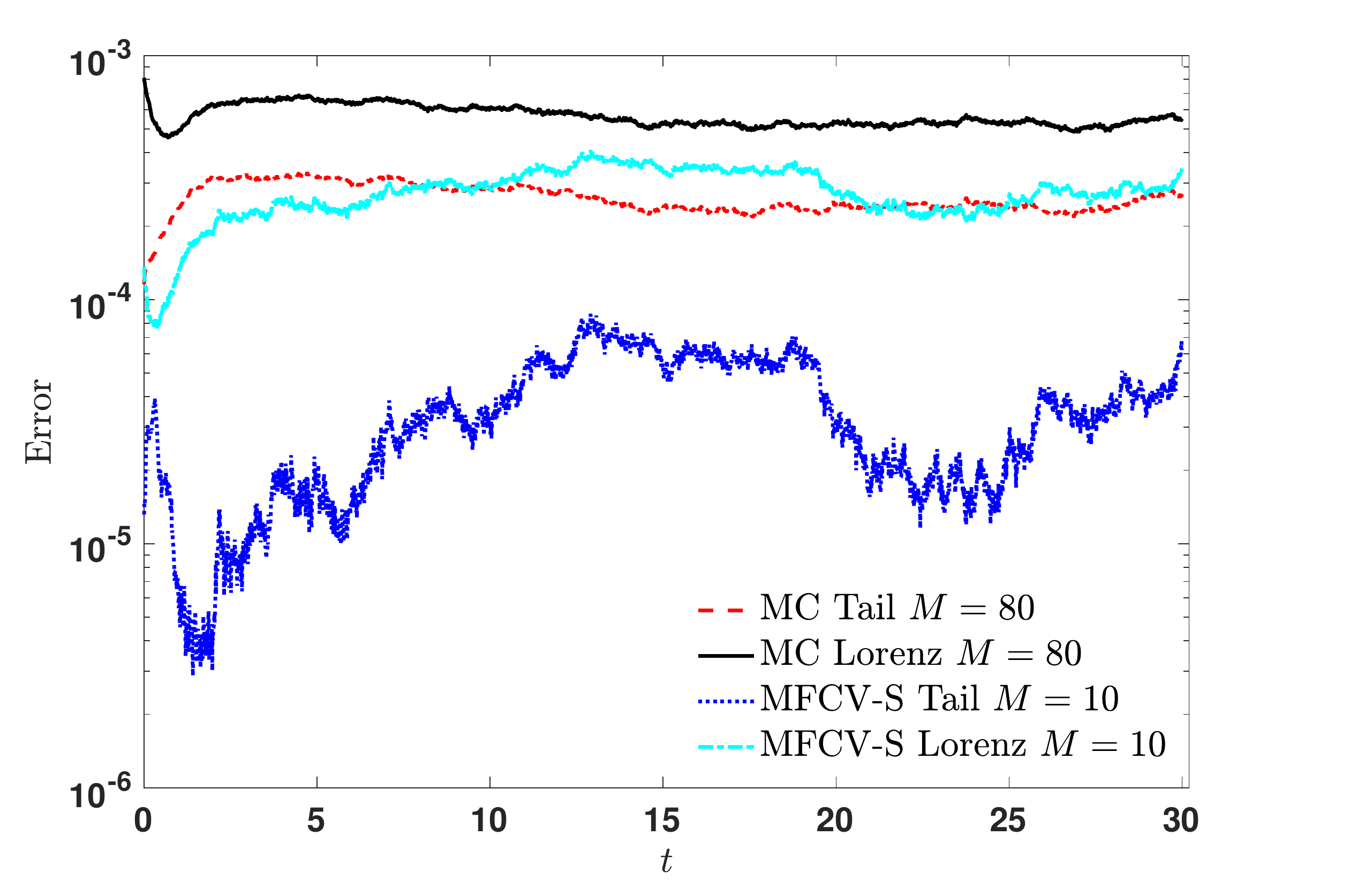}
\caption{\textbf{Test 2}. The $L_2$ error for the expected QoI computed for problem \eqref{eqn:A2} by the MC and MFCV-S method against time.  On the left we show the expected probability density function and Gini coefficient, whereas on the right the expected tail index and Lorenz curve. The DSMC solver has been implemented with $N=2\times 10^4$ particles.  }\label{CVTest2}
\end{center}
\end{figure}

Next, we consider uncertainty in the interaction 
\begin{equation}
\lambda(z) = \frac{1}{2}+ \frac{z}{4},\qquad z\sim\mathcal{U}([-1,1]),\qquad D(w)=w.
\tag{B}
\label{eqn:B2}
\end{equation}
The initial condition is uniformly distributed on $[0,2]$, so that the large time behavior of the Fokker-Planck model is given by \eqref{eq:steady_invgamma} with $m_{\h} \equiv 1$. 

In Figure \ref{SamplesTest4} we present the $L^2$ error for $\mathbb E[f]$ at final time $t=30$ computed by the MC and MFCV-S method for $M = 10$. We deduce that similar to the previous test cases the MFCV-S method improves the accuracy in comparison to the MC method between one and two orders of magnitude. 

\begin{figure}
\begin{center}
\includegraphics[scale = 0.35]{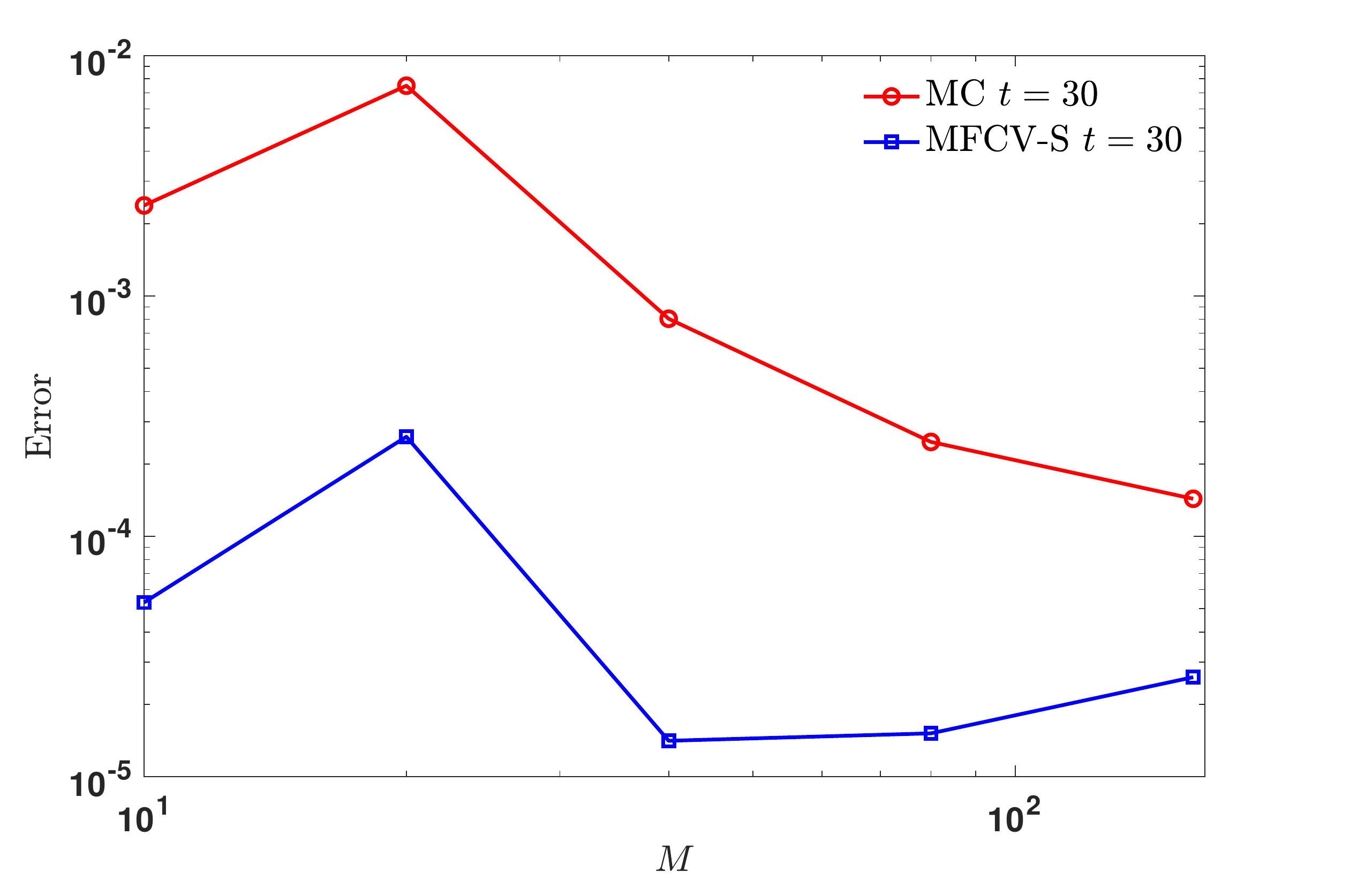}
\caption{\textbf{Test 2}. $L_2$ error for the expected distribution $\mathbb E[f_\epsilon]$ for problem \eqref{eqn:B2} computed by the  MC and MFCV-S method for different number of samples $M$ at $t=30$. We considered $N=2\times 10^{4}$ in the DSMC solver. } 
\label{SamplesTest4}
\end{center}
\end{figure}

\subsection{General Mean Field Control Variate (MFCV)}
The MFCV method relies on the numerical solution of the Fokker-Planck equation \eqref{FP} which is able to capture the transient behavior 
much better than the MFCV-S method. In fact the Fokker-Planck equation \eqref{FP} is computationally less expensive than the original kinetic model since the expensive integral operator is replaced by the simpler differential Fokker-Planck operator. In order to solve the mean field model we adopted the second order structure preserving method for nonlocal Fokker-Planck equations developed in \cite{pareschi2018structure}. Indeed, for our control variate method it is of immanent importance that the scheme is able to capture the asymptotic behavior of the model with arbitrary accuracy.  

We denote the numerical approximation of the mean field model at time $t^n$ by $\h^n_{\Delta w}$ and thus the control variate estimate for the QoI $q[f^n_\varepsilon]=q[f_\varepsilon(t^n)]$ becomes
 \begin{equation}
 E^{\lambda^*}_M[\g_N[f^n_\epsilon]]= E_M[ \g_N[f_\epsilon^{n}]]-  \lambda^{*,n}_M\ (E_M[\g[\h_{\Delta w}^{n}]]-\mathbb E[q[\h_{\Delta w}^{n}]]).
 \label{eq:mfcvg}
 \end{equation}
For the structure preserving method we select a semi implicit time integrator and a Gauss quadrature rule to evaluate the fluxes as defined in \cite{pareschi2018structure}. The boundary conditions are set to zero flux boundary conditions. As pointed out previously, the computational cost to evaluate \eqref{eq:mfcvg} is not negligible due to the time dependence of the control variate. In fact, in order to have an accurate estimate of $\mathbb E[q[\h_{\Delta w}^{n}]]$ the Fokker-Planck equation needs to be solved for a sufficiently large number of random samples $M_{MF}$ at each time step and estimate $\mathbb E[q[\h_{\Delta w}^{n}]]\approx E_{MF}[q[\h_{\Delta w}^{n}]]$. 
In the sequel, we discuss the issue of computational costs in more details.  
 
\subsubsection*{Estimating the computational cost} The costs of the solution of the Boltzmann equation at one time step can be estimated by the number of samples in physical space $N$ times the number of samples in random space $M$. We have
$$
\textrm{Cost}(E_M[\g_N[f_\epsilon]])= N\ M. 
$$
For the mean field model with $N_{MF}$ grid points and $M_{MF}$ samples in random space we get:
$$
\textrm{Cost}(E_{M_{MF}}[\g_{N_{MF}}[\h]])= N_{MF}\ M_{MF}. 
$$
Additionally, the total computation costs depend on the selected time step. The time step of the Boltzmann model is allowed to take value $\Delta t\in(0,\epsilon]$ in the DSMC solver. As in all our simulations we choose the maximum allowed value $\Delta t=\epsilon$. Consequently, since we use a semi-implicit method for the mean-field solver with time step $\Delta t_{MF}$ we have
$$
\epsilon=\Delta t = \frac{\Delta t_{MF}}{k},\ k\geq 1.
$$
We aim to control the cost of the mean field model by the total cost of the Boltzmann model, namely
$$
\textrm{Cost}(E_M[\g_N[f]])\geq \textrm{Cost}(E_{M_{MF}}[\g_{N_{MF}}[\h]]. 
$$
Thus, we have to ensure that 
$$
N\ M\geq \frac{N_{MF}\ M_{MF}}{k}
$$
is satisfied. Provided that the number of points and particles in physical space are fixed, we obtain the following upper bound on $M_{MF}$
\begin{align}
M_{MF}\leq \frac{k\ N\ M}{N_{MF}}. \label{uB} 
\end{align}
In practice, for a given $\epsilon$, $N$ and $M$, we define the number of grid points $N_{MF}$ in the deterministic mean-field solver which provides the largest value of the time step $\Delta t_{MF}$ to ensure stability. The maximum allowed number of samples in the mean-field model is then chosen accordingly to \eqref{uB}.

\subsubsection*{Test 3: Opinion model with uncertainty}
In order to show the performance of the MFCV method we select the same setting as defined in Section \ref{sect:opinion}.
The number of grid points of our mean field model is set to $N_{MF}=20$. Furthermore, the number of particles of the Boltzmann model is fixed to $N=2\times 10^4$ and we fix $k=1$. Hence, for $M_{MF}$ and arbitrary $M$ we get accordingly to \eqref{uB} the estimate
\begin{align}\label{upperBound}
M_{MF} \leq \frac{2\times 10^4\ M}{20}= 10^3\ M.
\end{align}
For our simulations we have chosen $M_{MF}=10^4$ which is accordingly to \eqref{upperBound} exactly the upper bound for $M=10$. Thus, the cost of the MFCV method are comparable to the cost of the corresponding MC method.\\

In Figure \ref{NumTest1} we compare the MFCV method with the MFCV-S for different numbers of samples $M$. The QoI is $\mathbb E[f_\epsilon]$. As expected we gain accuracy and the error of the MFCV method is below the error of the MFCV-S method in transient times, we considered here $t=0.1$.

\begin{figure}
\begin{center}
\includegraphics[width=0.5 \textwidth]{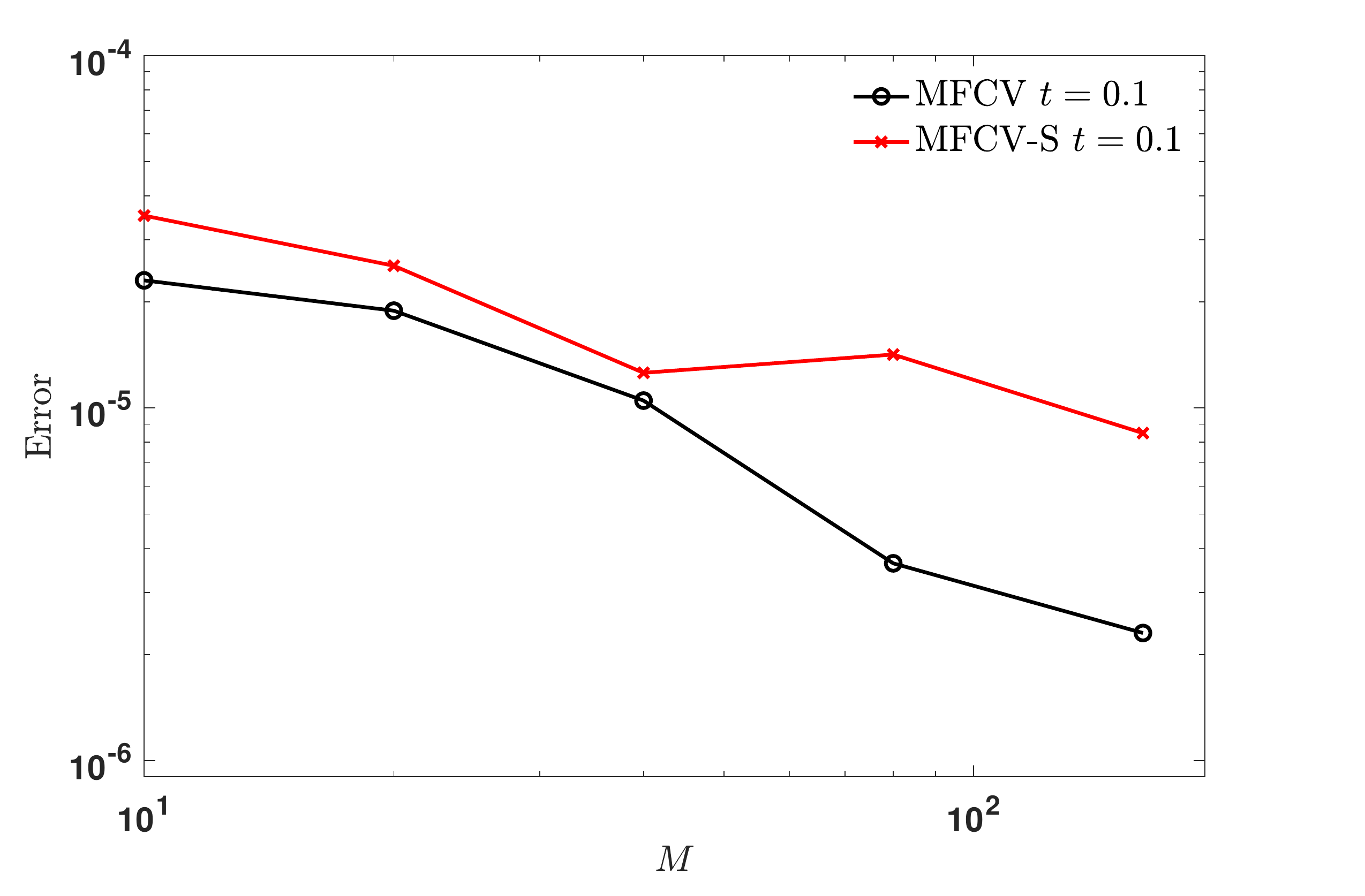}\hfill
\includegraphics[width=0.5 \textwidth]{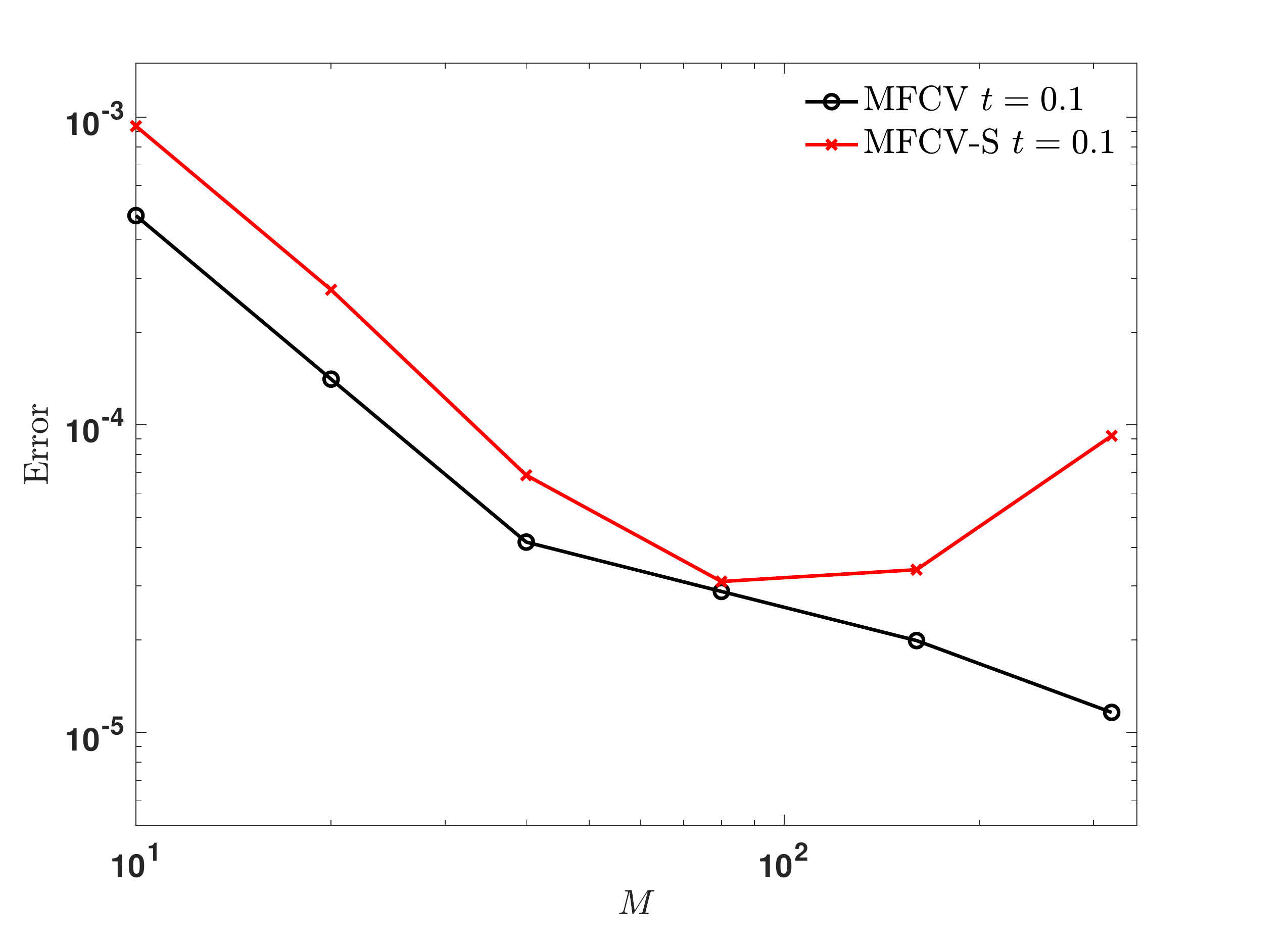}
\caption{ \textbf{Test 3}. The $L^2$ error of $\mathbb E[f]$ computed by the MFCV and MFCV-S method for different number of samples $M$ at $t=0.1$. We consider the same opinion formation model of Test 1. 
The left hand side corresponds to the setting of case \eqref{eqn:A} and the right to case \eqref{eqn:B}. }\label{NumTest1}
\end{center}
\end{figure}

\subsubsection*{Test 4: Wealth model with uncertainty}
For the CPT model we consider $N_{MF}=100$ and we choose $N=5\times 10^4$, $M_{MF}= 5\times 10^3$ and $k=1$. Thus the upper bound \eqref{uB} is satisfied for all samples $M\geq 10$. 

In Figure \ref{NumTest4} we compare the  $L^2$ error of $\mathbb E[f_\epsilon]$ computed by the MFCV or MFCV-S method at fixed times but for different number of samples $M$. 
We deduce that the MFCV method is able to be more accurate than the MFCV-S method also in short times. 
\begin{figure}
\begin{center}
\includegraphics[width=0.5 \textwidth]{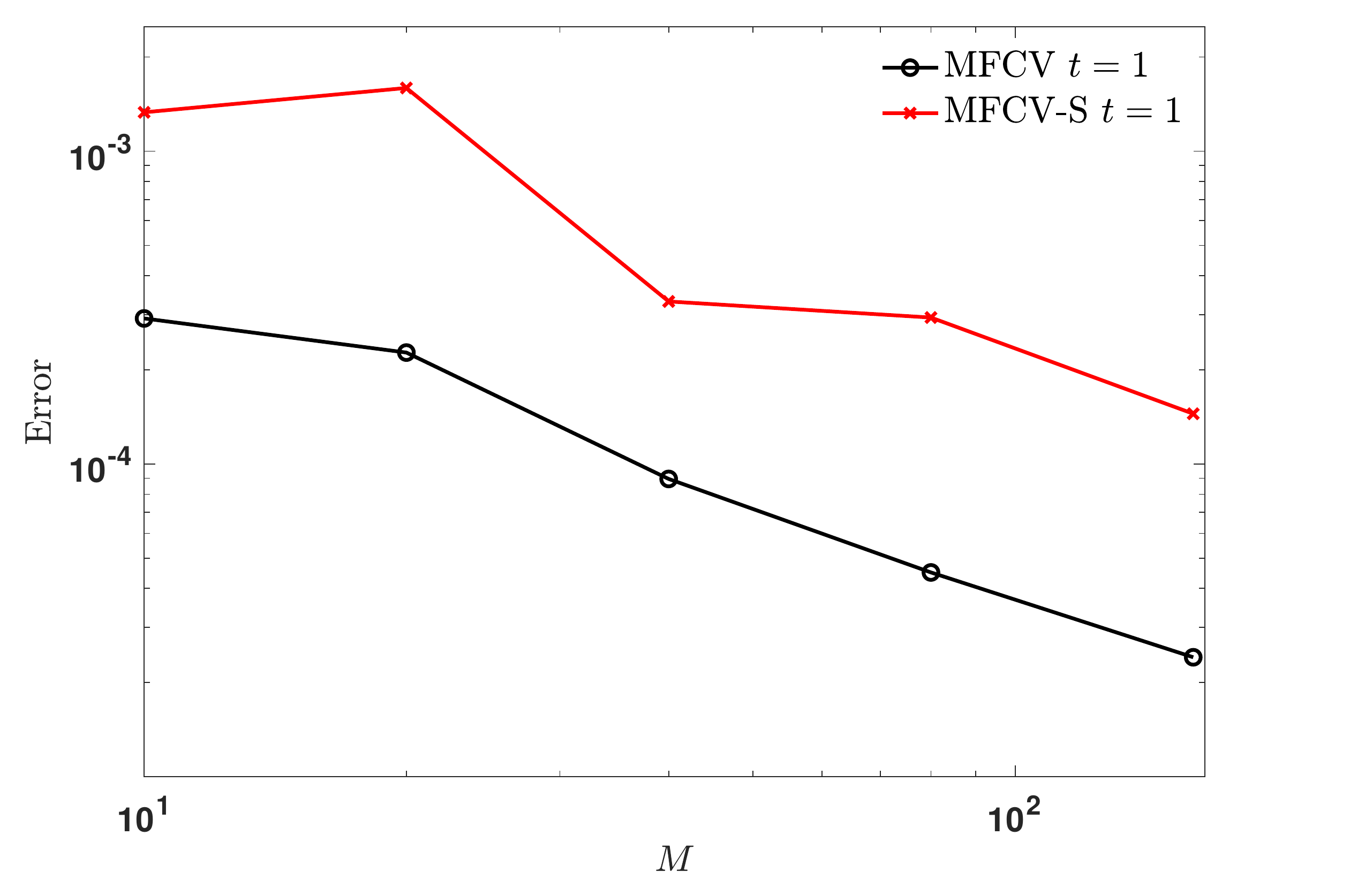}\hfill
\includegraphics[width=0.5 \textwidth]{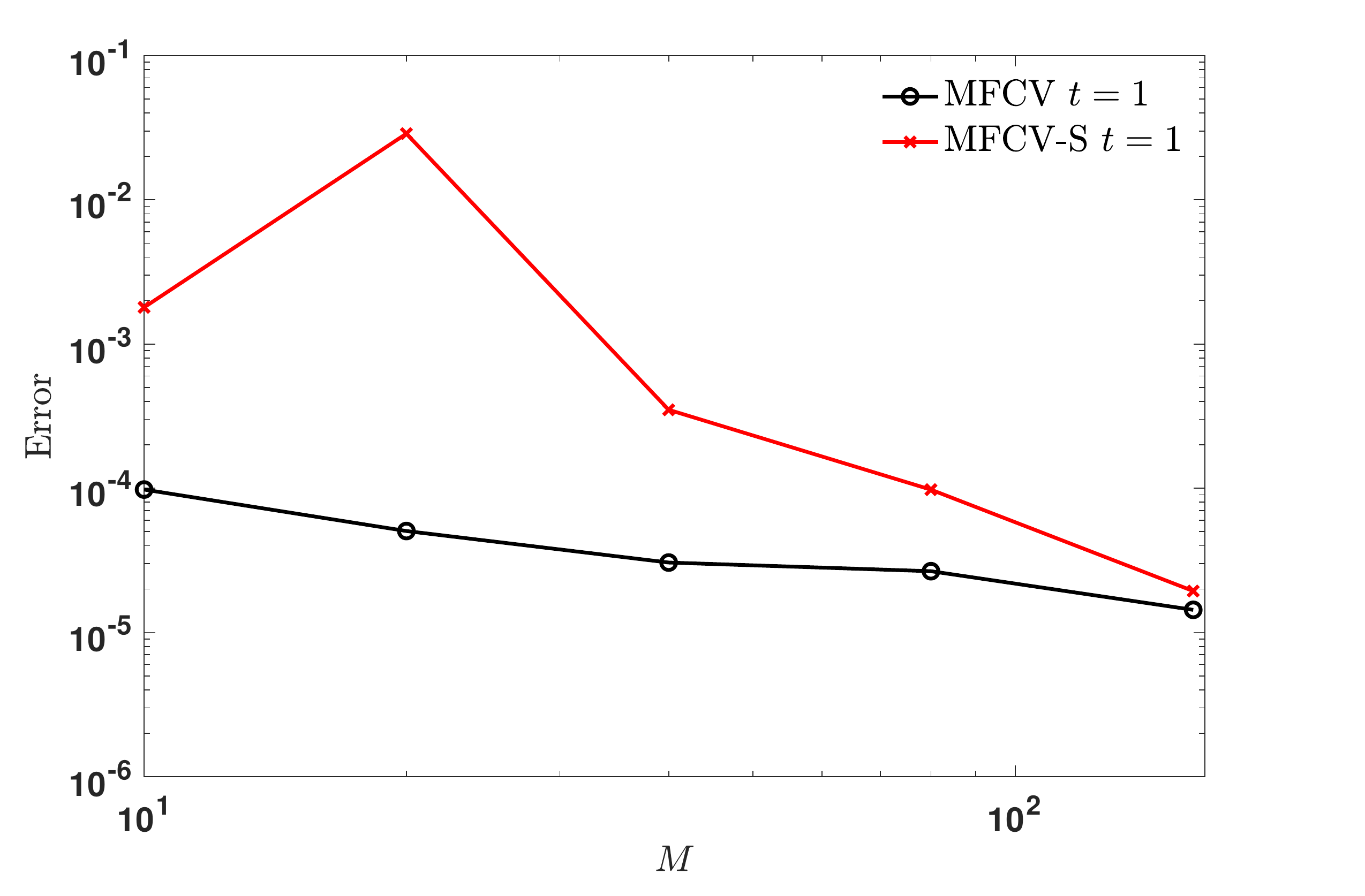}
\caption{\textbf{Test 4}. The $L^2$ error of $\mathbb E[f]$ computed by the MFCV and MFCV-S method for different number of samples $M$ at $t=1$. We consider the same wealth model of Test 2. 
The left hand side corresponds to the setting of case \eqref{eqn:A} and the right to case \eqref{eqn:B}.}\label{NumTest4}
\end{center}
\end{figure}

\subsubsection*{Test 5: Bounded confidence opinion model with uncertainty}
In this last test case we consider the opinion model introduced in section \ref{examples} with bounded confidence interaction rule \cite{HK,PTTZ}. 
The uncertainty is present in the function $P$ which weights the compromise tendency.  We consider 
\begin{equation}
P(|w-v|)=\chi(|w-v|< z),\quad  z\sim\mathcal{U}([1, 2]),\qquad D(w)=1-w^2.
\end{equation}
In comparison to the previous test cases the steady sate distribution of the mean field model is unknown. Therefore, we cannot apply the MFCV-S method but only the MFCV method.
The number of grid points of our mean field model is set to $N_{MF}=40$, the number of particles of the Boltzmann solver is fixed to $N=2\times 10^4$ and the number of samples for the mean field model is given by $M=5\times 10^3$. The simulations have been conducted with $\epsilon=2.5\times 10^{-4}$ and the initial distribution is given by 
\[
f(0,w) = C \exp\left\{-30\left(w+\frac12\right)^2\right\}+C \exp\left\{-30\left(w-\frac12\right)^2\Big)\right\},
\]
with normalization constant $C>0$. 

In Figure \ref{NumTest5} we report the $L^2$ error of the $\mathbb E[f]$ plotted for increasing number of samples $M$ (left plot) and for fixed number of samples in the time interval $[0,10]$. We deduce that the MFCV method in comparison to MC method achieves an improvement in accuracy between one and two orders of magnitude. In details, we obtain that the MFCV method with 10 samples performs even better than the MC method with 80 samples (right picture). 

\begin{figure}
\begin{center}
\includegraphics[width=0.5 \textwidth]{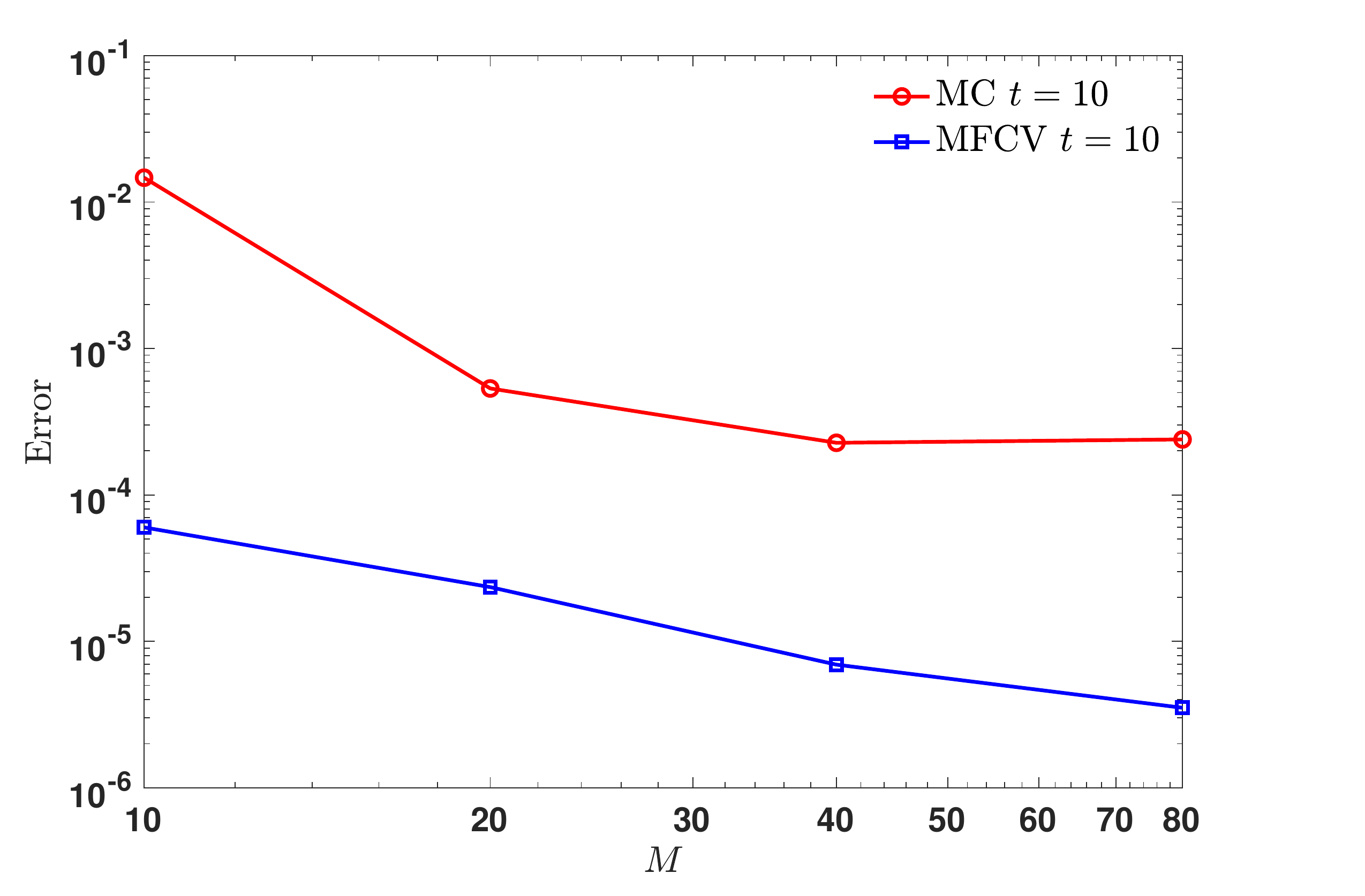}\hfill
\includegraphics[width=0.5 \textwidth]{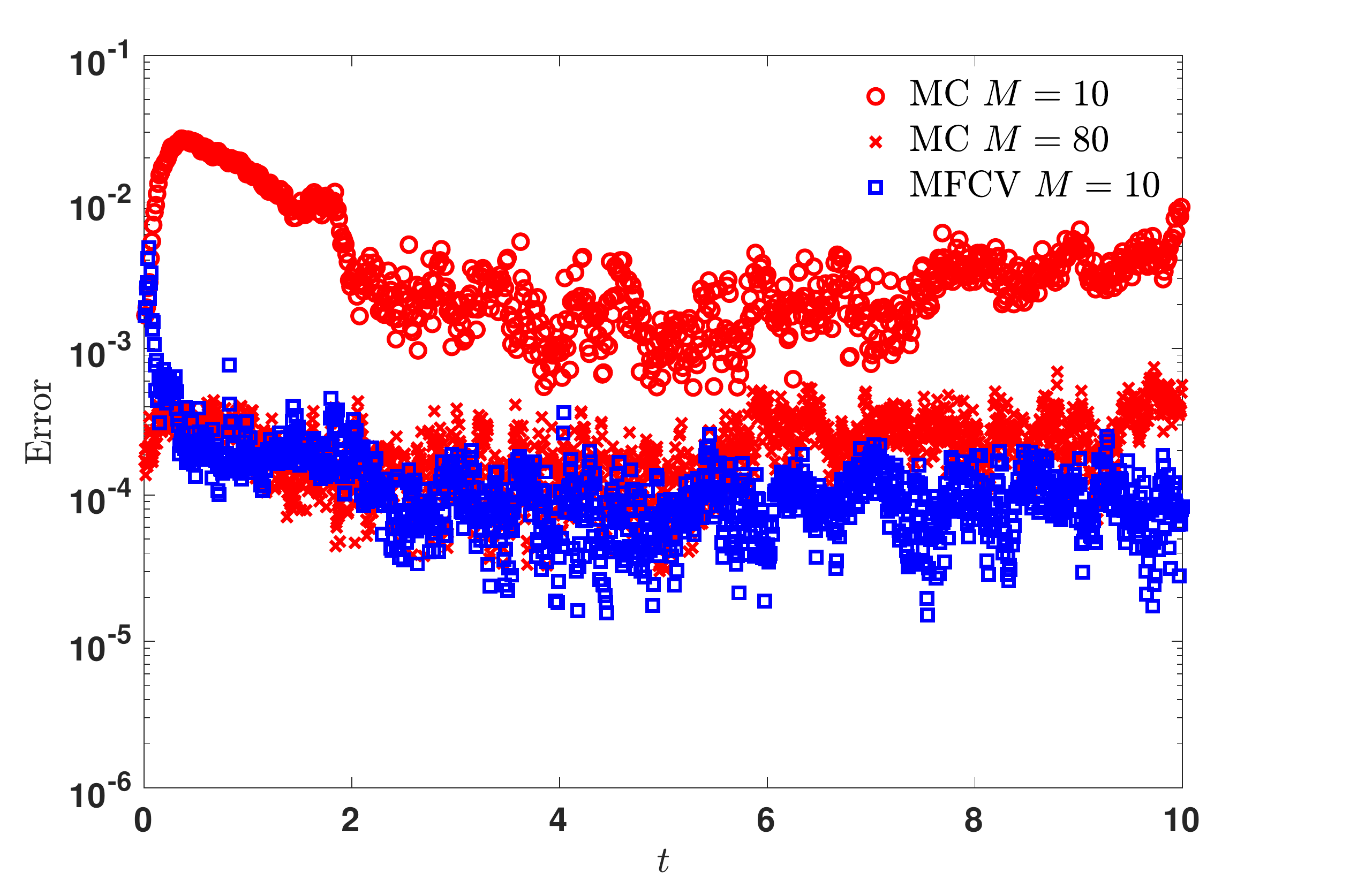}
\caption{\textbf{Test 5.} $L^2$ error for $\mathbb E[f]$ using the bounded confidence model of opinion, for different number of samples (left) and evolution in time (right). The simulation has been conducted with $N=2\times 10^4$ particles for the Boltzmann solver and the scaling $\epsilon=2.5\times 10^{-4}$. }\label{NumTest5}
\end{center}
\end{figure}

\section{Conclusion}
We introduced a novel variance reduction technique for uncertainty quantification in Boltzmann type equations
of interest in socio-economic applications. 
The method relies on a control variate strategy in order to speed up the convergence of the standard Monte Carlo method sampling method in the random space. In contrast to the classical Boltzmann equation of rarefied gas dynamics we cannot use as in \cite{dimarco2018multi} the knowledge of the local equilibrium state to design control variate methods. For this purpose, we considered the mean field approximation of the Boltzmann model as surrogate model for variance reduction. 
Therefore, the mean field control variate (MFCV) method makes use of the less expensive solution of the mean field approximation to accelerate the Monte Carlo convergence. In the physical space, the mean field model has been computed with a deterministic structure preserving method whereas the kinetic Boltzmann model has been solved by a standard DSMC method. Thus, the novel MFCV method can be regarded as a hybrid method in the physical space. Whenever the steady state of the mean field model is known it is possible to consider a simplified control variate strategy that uses the steady state as surrogate model. This latter approach, even if less accurate in transient regimes, leads to a strong computational cost reduction since the control variate can be evaluated off line. 
The numerical results confirm that the MFCV methods outperform the standard MC method in all applications considered. Although we have focused in this work on space-homogeneous Boltzmann type equations for socio-economic applications it is possible to apply the new MFCV method to a larger class of kinetic equations. Besides applications to other Boltzmann equations, several extensions are actually under study. Among these, the case of space non-homogeneous Boltzmann type equations and the introduction of multiple control variates as in \cite{dimarco2020multiscale}. 

\subsection*{Acknowledgments}
This research is funded by the Deutsche Forschungsgemeinschaft (DFG, German Research Foundation) under Germany's Excellence Strategy -- EXC-2023 Internet of Production -- 390621612 and supported also by DFG HE5386/15.
T. T. acknowledges the support by the ERS Prep Fund - Simulation and Data Science. The work was partially funded by the Excellence Initiative of the German federal and state governments. L. P. would like to thank the Italian Ministry of Instruction, University and Research (MIUR) to support this research with PRIN Project 2017, No. 2017KKJP4X - Innovative numerical methods for evolutionary partial differential equations and applications.
 This paper was written within the activities of the GNFM and GNCS of INDAM. 
The research was partially supported by the Italian Ministry of University and Research (MUR):  Dipartimenti di Eccellenza Program (2018--2022) - Dept. of Mathematics "F. Casorati", University of Pavia.

%-- LITERATUR ----------------------------------------------------------%
%	\bibliography{literaturmean.bib}
%		\bibliographystyle{abbrv}	
%	
%	

\end{document}